\documentclass[a4paper,12pt]{article}

\usepackage{latexsym}
\usepackage{amsmath}
\usepackage{amsthm}
\usepackage{amsbsy}
\usepackage{amssymb}
\usepackage[pdftex]{graphicx}
\usepackage{epstopdf}
\usepackage{mathrsfs}
\usepackage{times}
\usepackage{afterpage}
\usepackage[prefix]{nomencl}
\usepackage{psfrag}
\usepackage{multido}
\usepackage{algorithm} 
\usepackage{algpseudocode}
\usepackage{mathabx}

\usepackage{amsfonts}
\usepackage{arydshln}
\usepackage{booktabs}	
\usepackage{url}
\usepackage{hyperref} 

\input epsf

\renewcommand{\vec}[1]{\mbox{\boldmath$#1$}}

\newcommand{\dif}{\mathrm{d}}

\usepackage{color}

\usepackage{amsthm}
\newtheorem{theorem}{Theorem}[section]
\newtheorem{remark}[theorem]{Remark}
\newtheorem{definition}[theorem]{Definition}
\newtheorem{lemma}[theorem]{Lemma}
\newtheorem{corollary}[theorem]{Corollary}
\allowdisplaybreaks[3]

\DeclareMathOperator*{\esssup}{ess \, sup}

\newcommand{\im}{\mathrm{i}}

\begin{document}
\begin{center}
{\Large \bf 
Generalised Magnetic Polarizability Tensors}
\\
$^*$P.D. Ledger and $^\dagger$W.R.B. Lionheart \\
$^*$Zienkiewicz Centre for Computational Engineering, College of Engineering, \\Swansea University Bay Campus, Swansea. SA1 8EN\\
$^\dagger$School of Mathematics, Alan Turing Building, \\The University of Manchester, Oxford Road, Manchester. M13 9PL\\
13th January 2018
\end{center}

\section*{Abstract}
We present a new complete asymptotic expansion for the low frequency time-harmonic magnetic field perturbation caused by the presence of a conducting (permeable) object as its size tends to zero for the eddy current regime of Maxwell's equations. The new asymptotic expansion allows the characterisation of the shape and material properties of such objects by a new class of generalised magnetic polarizability tensors and we provide an explicit formula for their calculation. Our result will have important implications  for metal detectors since it will improve small object discrimination and, for situations where the background field varies over the inclusion, this information will be useable, and indeed useful, in identifying their shape and material properties. Thus, improving the ability of
metal detectors to locate landmines and unexploded ordnance, sort metals in recycling processes, ensure food safety as well as enhancing security screening at airports and public events.

\vspace{0.1in}

\noindent {\em MSC: 35R30, 35B30}

\vspace{0.1in}

\noindent {\bf Keywords:} Polarizability Tensors; Asymptotic Expansion; Eddy Currents; Metal Detectors; Land Mine Detection

\section{Introduction}

The characterisation of highly conducting objects from low frequency magnetic field perturbations has important applications in metal detection where the goal is to locate and identify concealed inclusions in an otherwise low conducting background. Metal detectors are used in the search for artefacts of archaeological significance, the detection of landmines and unexploded ordnance, the recycling of metals, ensuring food safety as well as in security screening at airports and at public events. The ability to better characterise objects offers considerable advantages in reducing the number of false positives in metal detection and, in particular, to accelerate and improve the accuracy of object location and discrimination.

For a range of electromagnetic and acoustic phenomena similar findings have been found where, in each case, an asymptotic expansion of the field perturbation caused by the presence of an inclusion as its size, $\alpha $, tends to zero results in formula which permits the low-cost characterisation of an object. In particular, in electromagnetics,
the expansion has been found to be of the form 
\begin{equation}
( {\vec U}_\alpha-{\vec U}_0 )( {\vec x} )_i = ( {\vec D}_{x}^2 G( {\vec x} , {\vec z} ) )_{ij} {\mathcal A}_{jk} ( {\vec U}_0 ({\vec z}))_k +( {\vec R} )_i \label{eqn:basicform},
\end{equation}
when orthonormal coordinates and Einstein summation convention is used. In this expression $({\vec U}_\alpha - {\vec U}_0)({\vec x})$ represents the perturbed field at location ${\vec x}$, the object is assumed to have the form $B_\alpha =\alpha B+{\vec z}$, where $B$ is its shape and ${\vec z}$ denotes its position, ${\vec U}_0({\vec x})$ is the background field, $ {\vec D}_x^2 G({\vec x},{\vec z})$ is the Hessian of an appropriate Green's function, ${\vec R}$ is a residual term and ${\mathcal A}$ is a symmetric rank 2 polarizability tensor. The polarizability tensor is independent of ${\vec z}$ and is a function of $B$, and hence its topology~\footnote{ The zeroth Betti number $\beta_0(B)$, is the number of connected parts of $B$, which for a bounded connected region in ${\mathbb R}^3$ is always $1$. The first Betti number, $\beta_1(B)$ is the genus, i.e. the number of handles and the second Betti number $\beta_2(B)$ is one less than the connected 
parts of the boundary $\partial B$, ie. the number of cavities (e.g.~\cite{bossavit1998}). },
it is also a function of the object's material characteristics and, thus, provides a means for its characterisation. {Importantly, (\ref{eqn:basicform}) separates an object characteristics from the background field and, consequently, has applications in the low-cost identification of hidden targets in inverse problems~\cite{ammarikangbook, ammarikanglecturenotes}.}

 Explicit formulae for the calculation of polarizability tensors have been found in a range of different electromagnetic applications. These include the leading order term in an expansion of the perturbed magnetic field $({\vec H}_\alpha-{\vec H}_0)({\vec x})$ as $\alpha \to 0$ for a (multiply connected) permeable object with $\beta_0(B)=1,  \beta_1 (B)\ge 0, \beta_2(B)\ge 0$ in magnetostatics~\cite{ledgerlionheart2016} and expansions of   $({\vec H}_\alpha-{\vec H}_0)({\vec x})$, and of the perturbed electric field $({\vec E}_\alpha-{\vec E}_0)({\vec x})$, as $\alpha \to 0$ in electromagnetic scattering by simply connected dielectric, permeable or conducting objects with $\beta_0(B)=1, \beta_1(B)=\beta_2(B)=0$~\cite{ammarivogeluisvolkov,ammarivolkov2005,ledgerlionheart2012}.
 In the aforementioned cases, ${\mathcal A}={\mathcal T}(c)$ is a suitably parameterised rank 2 P\'oyla-Szeg\"o tensor and its coefficients can be computed by solving $3$ scalar transmission problems or through the solution of  $3$ integral equations~\cite{ammarikangbook} for  given $\alpha$, $B$ and material contrast $c$.

More recently, for the metal detection problem, Ammari, Chen, Chen, Garnier and Volkov~\cite{ammarivolkov2013} have obtained the leading order term in an expansion of  $({\vec H}_\alpha-{\vec H}_0)({\vec x})$   as $\alpha \to 0$ for a highly conducting (multiply connected) object placed in a low-frequency time harmonic background field, ${\vec H}_0({\vec x})$. This expansion was obtained for the eddy current regime of the Maxwell system, which is the relevant case for metal detection, and they showed that the object is characterised by a rank 4 tensor and, therefore,  is of a different form to (\ref{eqn:basicform}). However, for the case of orthonormal coordinates, we have shown that their leading order term does reduce to the same form as in (\ref{eqn:basicform}) and, in this case,  ${\mathcal A}=\widecheck{\widecheck{\mathcal M}}$ is a new complex symmetric rank 2 tensor~\cite{ledgerlionheart2014}. In this notation, a single check indicates a reduction in a tensor's rank by one, which is achieved due to skew symmetry of the tensor's coefficients in two of its indices, a double check indicates a reduction in a tensor's rank by two. The coefficients of this tensor are computed by solving $3$ vectorial transmission problems and are a function of $B$, $\alpha$, the object's conductivity, $\sigma_*$, its permeability contrast, $\mu_*/\mu_0$,  as well as the frequency of excitation, $\omega$. This result also provides a solid mathematical footing for denoting $\widecheck{\widecheck{\mathcal M}}$ as the rank 2 magnetic polarizability tensor (MPT), which the electrical engineering community predict for describing the characteristics of a conducting object in metal detection e.g.~\cite{norton2001,barrowes2004,peyton2013}. 
 In~\cite{ledgerlionheart2016}, we have  obtained further results which relate the coefficients of $\widecheck{\widecheck{\mathcal M}}$ to  ${\mathcal T}(\mu_*/\mu_0)$ in the limiting case of $\omega \to 0$, and, independent of the value of $\beta_1(B)$. For the limiting case of $\sigma_* \to \infty$, we have also shown that the coefficients of $\widecheck{\widecheck{\mathcal M}}$ tend to those of ${\mathcal T}(0)$ if $\beta_1(B)=0$. These results allow permeable and non--permeable objects to be distinguished and some topology information to be extracted. Furthermore, we have computed $\widecheck{\widecheck{\mathcal M}}$ for a range of simply and multiply connected objects using a $hp$-finite element approach and explored how their coefficients vary over a range of frequencies, within the validity of the eddy current model~\cite{ledgerlionheart2016}.

Although the leading order term in the expansion of $({\vec H}_\alpha-{\vec H}_0)({\vec x})$ as $\alpha \to 0$ and $\widecheck{\widecheck{\mathcal M}}$ provide useful information about an object, they impose limitations due to the assumption that ${\vec H}_0({\vec x})$ is uniform over the inclusion, with it only being evaluated at ${\vec z}$, and that $\widecheck{\widecheck{\mathcal M}}$ has at most $6$ independent complex coefficients.
In practical magnetic induction metal detection and testing, however, the ${\vec H}_0({\vec x})$ generated by coil arrays is significantly non-uniform over the object unless the distance from the coils is very large. For instance,  this is the case in walk--through metal detectors, when there is little space between the coil arrays and the person being tested for a security threat, and in subsurface metal detection, when a metallic object is buried close to the surface. In such situations, the leading order term in the expansion of  $({\vec H}_\alpha-{\vec H}_0)({\vec x})$ as $\alpha \to 0$ will not accurately describe the field perturbation and $\widecheck{\widecheck{\mathcal M}}$ will not provide an accurate object characterisation.
Still further, there are difficulties in separating geometrical information from the material contrast $c$ in ${\mathcal T}(c)$~\cite{ammarikangbook}, and hence the limiting cases of $\widecheck{\widecheck{\mathcal M}}$.   
Finally, if an object has rotational or reflectional symmetries the number of non-zero independent coefficients in a symmetric rank 2 tensor are greatly reduced~\cite{ledgerlionheart2014} making discrimination between objects difficult (eg.
the independent non-zero coefficients of $\widecheck{\widecheck{\mathcal M}}$ for a cylinder and a cone are the same {due to rotational and reflectional symmetries that are present in both objects, even though the cylinder has an additional mirror symmetry normal to a rotation axis that is not present in a cone}).

In order to describe the field perturbation more accurately, and better characterise a conducting permeable object, higher order terms in the asymptotic expansion are required.
 For the electrical impedance tomography (EIT) problem, where the perturbed electric field due to the presence of a small conducting inclusion can be described in terms of the gradient of a scalar potential, Ammari and Kang have obtained a complete asymptotic expansion as $\alpha \to 0$~\cite{ammarikang2003,ammarikangbook}. Here, the object is described by generalised polarizability tensors (GPTs) with the lowest order case agreeing with the rank 2 tensor ${\mathcal T}(c)$. This class of GPTs satisfies symmetry conditions on the space of harmonic polynomials. Complete asymtopic expansions of the perturbed field for a small object have also been obtained for acoustics~\cite{ammarikanghelm} and the elasticity problem~\cite{ammarikangbook} where the object is again characterised by GPTs.

In this work, we provide a new complete asymptotic expansion of $({\vec H}_\alpha- {\vec H}_0)({\vec x})$ for a highly conducting (possibly permeable and multiply connected) object as $\alpha \to 0$. Thus, extending the result in~\cite{ammarivolkov2013,ledgerlionheart2014}, which provided only the leading order term. We write our result in terms of a new class of (higher rank) generalised magnetic polarizability tensors (GMPTs), which characterise the object's shape and its material characteristics. 
The GMPTs we obtain are quite different to the GPTs previously presented by Ammari and Kang. The explicit expression for their coefficients are with respect to the standard orthonormal basis rather than the space of harmonic polynomials. 
They are functions  of $B$,  $\alpha$, $\sigma_*$, $\mu_*/\mu_0$, $\omega$ and can be computed by solving a generalised form of the vectorial transmission problem obtained in~\cite{ammarivolkov2013,ledgerlionheart2014}. Moreover, the leading order term in our new expansion agrees with our previous result~\cite{ledgerlionheart2014} and here the GMPT agrees with  $\widecheck{\widecheck{\mathcal M}}$. 
Our new complete expansion will overcome the aforementioned difficulties of  just using the leading order term for $({\vec H}_\alpha- {\vec H}_0)({\vec x})$  as $\alpha \to 0$ and describing the object using $\widecheck{\widecheck{\mathcal M}}$ when  ${\vec H}_0({\vec x})$ is non--uniform over the object, such as in a walk--through metal detector for a security threat and in subsurface metal detection for a metallic object buried close to the surface. 
For such applications, it will improve the accuracy of $({\vec H}_\alpha-{\vec H}_0)({\vec x})$, by including more terms in the expansion, and will improve the characterisation of an object's shape and its material properties, by the increased number of independent parameters in the GMPTs. Consequently, improving object identification and location.

The paper is organised as follows: In Section~\ref{sect:problemconfig}, the problem configuration is briefly described and some notation is introduced for the presentation of our new results.  Section~\ref{sect:leadingterm} summaries the previously known results about  $\widecheck{\widecheck{\mathcal M}}$ and the leading order term of $({\vec H}_\alpha- {\vec H}_0)({\vec x})$  as $\alpha \to 0$ due to the presence of a highly conducting object~\cite{ammarivolkov2013,ledgerlionheart2014}. Section~\ref{sect:mainresult} states our new main result and Section~\ref{sect:derive} contains the lemmas associated with the derivation of our asymptotic formula. Finally, Section~\ref{sect:tensorrep} is concerned with the representation of the asymptotic formula in terms of a new class of higher order GMPTs.

\section{Problem configuration} \label{sect:problemconfig}
The problem configuration has already been described in~\cite{ammarivolkov2013,ledgerlionheart2014} and is briefly recalled.
We consider an electromagnetic inclusion in ${\mathbb R}^3$ of the form $B_\alpha ={\vec z} + \alpha B$, where $B \subset {\mathbb R}^3$ is a bounded, smooth domain. Let $\Gamma$ and $\Gamma_\alpha$ denote the boundary of $B$ and $B_\alpha$, respectively, and $\mu_0$ the permeability of free space. We continue to follow the previous notation and write
\begin{equation}
\mu_\alpha=\left \{ \begin{array}{ll}
\mu_* & \hbox{in $B_\alpha$} \\
\mu_0 & \hbox{in ${\mathbb R}^3 \setminus B_\alpha$} \end{array} \right . ,
\sigma_\alpha=\left \{ \begin{array}{ll}
\sigma_* & \hbox{in $B_\alpha$} \\
0 & \hbox{in ${\mathbb R}^3 \setminus B_\alpha$} \end{array} \right . , \label{eqn:matparam}
\end{equation}
where $\mu_*$ and $\sigma_*$ denote the object's permeability and conductivity, respectively, which we assume to be constant. 
The time harmonic fields ${\vec E}_\alpha$ and ${\vec H}_\alpha$ that result from a {compactly supported} time varying current source, ${\vec J}_0$, located away from $B_\alpha$ and satisfying $\nabla\cdot{\vec J}_0=0$ in ${\mathbb R}^3$, and their interaction with the object $B_\alpha$, satisfy the eddy current equations in a weak sense~\cite{ammaribuffa2000}
\begin{subequations}
\begin{align}
\nabla \times {\vec E}_\alpha & = \im \omega \mu_\alpha {\vec H}_\alpha && \hbox{in ${\mathbb R}^3$} ,\\
\nabla \times {\vec H}_\alpha & = \sigma_\alpha {\vec E}_\alpha +{\vec J}_0
&& \hbox{in ${\mathbb R}^3$} ,\\
{\vec E}_\alpha ({\vec x}) = O(|{\vec x}|^{-1}),  \, &   {\vec H}_\alpha ({\vec x}) = O(|{\vec x}|^{-1}) && \hbox{as $|{\vec x}| \to \infty $} ,
\end{align} \label{eqn:eddymodel}
\end{subequations}
where $\omega$ denotes the angular frequency and $\im:=\sqrt{-1}$. Letting $\alpha =0$ in (\ref{eqn:eddymodel}) we obtain the corresponding fields, ${\vec E}_0$ and ${\vec H}_0$, that result from time varying current source in the absence of an object.  As explained in~\cite{ammaribuffa2000}, the eddy current model is completed by $\nabla \cdot {\vec E}_\alpha =0$ in ${\mathbb R}^ 3 \setminus B_\alpha$ and the uniqueness of ${\vec E}_\alpha$ in ${\mathbb R}^3 \setminus B_\alpha$ is achieved by additionally specifying
\begin{equation}
\int_{\Gamma_\alpha} {\vec n} \cdot {\vec E}_\alpha |_+ \dif {\vec x} = 0, \label{eqn:efieldunique}
\end{equation}
where ${\vec n}$ is the outward normal to $\Gamma _\alpha$. Furthermore, in practice, the decay of the fields is actually faster than the $|{\vec x}|^{-1}$ stated in the original eddy current model~\cite{ammaribuffa2000}.

The task is to develop an asymptotic expansion for $ ({\vec H}_\alpha  - {\vec H}_0) ({\vec x} )$ as $\alpha \to 0$ for the case where 
\begin{equation}
\nu : =k \alpha^2=O(1) , \qquad k:=\omega \mu_0 \sigma_*, \label{eqn:nudefine}
\end{equation}
which includes the case of fixed $\sigma_*,\mu_*, \omega$ as $\alpha \to 0$ (since in this case $|\nu| \le C \alpha^2 \le C$). Notice that the condition on $\nu$  is required to ensure the eddy current model is not violated as the object size vanishes.

For what follows  it is beneficial to introduce the following notation:
\begin{definition} 
We will use boldface for vector quantities  (e.g. ${\vec u}$) and denote by ${\vec e}_j$, $j=1,2,3$ the units vectors associated with an orthonormal coordinate system. We denote the $j$-th component of a vector ${\vec u}$ in this coordinate system by $({\vec u})_j = {\vec u} \cdot {\vec e}_j = u_j$. 
\end{definition}

\begin{definition}
We will use calligraphic symbols to denote rank 2 tensors  e.g. ${\mathcal N} = {\mathcal N}_{ij} {\vec e}_i \otimes {\vec e}_j$ and denote their coefficients by ${\mathcal N}_{ij}$. 
\end{definition}

\begin{definition}
By symbols in the Fraktur font, e.g. ${\mathfrak A}$, we shall denote higher order tensors and, to describe their coefficients with respect to an orthonormal coordinate basis, it will be useful to introduce  the $p$-tuple of positive integers $J(p):=[j_1 , j_2, \cdots , j_p]$, the $m$-tuple of positive integers  $K(m):=[k_1 , k_2, \cdots ,k_m]$ and to introduce the $(p+1)$-- and $(m+1)$--tuple of positive integers ${J}(p+1) =: [ j, J(p)]$ and ${K}(m+1)=[k,K(m)]$, respectively. Thus by ${\mathfrak N}_{ [  j,{J}(p), k,{K}(m) ] } ={\mathfrak N}_{  {J}(p+1) {K}(m+1)  } $~\footnote{When no confusion arises, we will drop the square parentheses on  the lists of  indices} we denote the coefficients of the rank $2+ p+m$ tensor 
\begin{equation}
\displaystyle {\mathfrak N} = {\mathfrak N}_{  {J}(p+1) {K}(m+1)  } {\vec e}_j  \otimes \left ( \bigotimes_{\ell=1}^p  {\vec e}_{j_\ell}  \right ) \otimes {\vec e}_k \otimes  \left ( \bigotimes_{\ell=1}^m  {\vec e}_{k_\ell}  \right ) . \nonumber
\end{equation}
 For $p= m=0$ this reduces  to the rank 2 tensor ${\mathcal N} = {\mathcal N}_{kj} {\vec e}_k \otimes {\vec e}_j = {\mathfrak N}_{kj}  {\vec e}_k \otimes {\vec e}_j$. Consider the rank $4+m+p$ tensor 
 \begin{align}
\displaystyle {\mathfrak A} = & {\mathfrak A}_{ [  h,i,k,K(m) , j, J(p)   ] }  {\vec e}_h\otimes  {\vec e}_i  \otimes {\vec e}_k \otimes  \left ( \bigotimes_{\ell=1}^m {\vec e}_{k_\ell}  \right ) \otimes
 {\vec e}_j \otimes \left ( \bigotimes_{\ell=1}^p  {\vec e}_{j_\ell}  \right ) , \nonumber 
\end{align}
{often we will write $ {\mathfrak A}_{ [  h,i,k,K(m) , j, J(p)   ] }= {\mathfrak A}_{ [  [h,i,k,K(m) ], [j, J(p)]   ] }= {\mathfrak A}_{ [  [h,i,K(m+1) ], J(p+1)   ] }$ to group indices and assist when considering products with other terms as well as when considering the skew symmetries of this tensor. However, by the introduction of such brackets, we do not imply skew systematization over these indices. Using skew symmetries,}
we will denote by a single check  (i.e. $\widecheck{\displaystyle {\mathfrak A} }$) a reduction in rank by $1$ and by a double check (i.e. $\widecheck{\widecheck{\displaystyle {\mathfrak A} }}$) a reduction in its rank by $2$.
\end{definition}

\begin{definition} Let
 $\displaystyle(\Pi({\vec \xi}))_{J(p)}: =\prod_{\ell=1}^p \xi_{j_\ell} =({\vec \xi})_{j_1} ({\vec \xi})_{j_2} \cdots ({\vec \xi})_{j_p}= \xi_{j_1} \xi_{j_2} \cdots \xi_{j_p}$ and $\displaystyle(\Pi({\vec \xi}))_{K(m)}: =\prod_{\ell}^m \xi_{k_\ell}= \xi_{k_1} \xi_{k_2} \cdots \xi_{k_m}$ where ${\vec \xi}\cdot {\vec e}_j= \xi_j$ is the $j$th spatial coordinate measured from an origin contained in $B$. 
Furthermore, when $p=0$  (or $m=0$ ) then $J=\emptyset$  (respt. $K=\emptyset$) and, in this case, $(\Pi({\vec \xi}))_\emptyset=1$. 
\end{definition}

Using this notation,  we shall imply the Einstein summation convention for repeated sets of indices as appropriate.

{ 
\begin{definition}
We recall that for $0\le \ell < \infty$, $0\le p< \infty$,
\begin{equation}
\| {\vec u} \|_{W^{\ell ,p}(B_\alpha)} : = \left ( \sum_{j=0}^\ell  \int_{B_\alpha} | {\vec D}^j( {\vec u}({\vec x}))|^p \dif {\vec x} \right)^{1/p}, \nonumber 
\end{equation}
where the derivatives are defined in a weak sense and 
\begin{equation}
\| {\vec u} \|_{W^{\ell ,\infty}(B_\alpha)} : =\esssup_{{\vec x} \in B_\alpha} \sum_{j=0}^\ell | {\vec D}^j( {\vec u}({\vec x}))| \nonumber .
\end{equation}
\end{definition}
}

\section{Leading order term of  $ ({\vec H}_\alpha  - {\vec H}_0) ({\vec x} )$ as $\alpha \to 0$} \label{sect:leadingterm}
In~\cite{ammarivolkov2013}, Ammari {\em et al.} have obtained the leading order term in an asymptotic expansion of $ ({\vec H}_\alpha  - {\vec H}_0) ({\vec x} )$ as $\alpha \to 0$ for $\nu =O(1)$ and ${\vec x}$  away from the location ${\vec z}$ of the inclusion.   In~\cite{ledgerlionheart2014} we have previously shown that their result can be conveniently expressed using Einstein summation notation and, in the case of orthonormal coordinates,  that it reduces to
\begin{align}
({\vec H}_\alpha -& {\vec H}_0) ({\vec x} )_i =  ({\vec D}_{ x}^2 G( {\vec x} , {\vec z} ))_{ik} \widecheck{\widecheck{\mathcal M}}_{kj} ({\vec H}_0  ( {\vec z} ))_j +({\vec R}({\vec x}))_i  \label{eqn:rank2asymform},
\end{align}
with $|{\vec R}({\vec x})| \le C \alpha^{4}\| {\vec H}_0 \|_{W^{2,\infty}(B_\alpha)}$. In the above, $({\vec D}_x^2 G( {\vec x} , {\vec z}) )_{ik}$ are the components of the rank 2 tensor
${\vec D}_x^2 G( {\vec x} , {\vec z} ) = 1/( 4 \pi |{\vec r}|^3 ) ( 3\hat{\vec r}\otimes \hat{\vec r} - {\mathbb I}) =( {\vec D}_x^2 G( {\vec x} , {\vec z} ))_{ik} {\vec e}_i \otimes {\vec e}_k $. This is obtained  from $ ({\vec D}_x^2 G( {\vec x}, {\vec z}) )_{ik} : = \partial_{x_i} \partial_{x_k}(G( {\vec x}, {\vec z}) )$ where $G({\vec x},{\vec z}) : =1/(4 \pi |{\vec x}- {\vec z}|)= 1/(4 \pi |{\vec r}|)$, ${\vec r}= {\vec x}-{\vec z}$, $\hat{\vec r}= {\vec r}/|{\vec r}|$ and $({\mathbb I)}_{ik}=\delta_{ik}$ are the components of the identity tensor, which are equal to the Kronceker delta.
Furthermore, we have shown that $\widecheck{\widecheck{\mathcal M}}_{kj} := - \widecheck{\mathcal C}_{kj} + {\mathcal N}_{kj}$  are the coefficients of a  complex symmetric {\em magnetic polarizability} (MPT) rank 2 tensor,  which describes the shape, conductivity, and permeability (contrast) of the object, and is computed using
\begin{subequations}
\begin{align} 
\widecheck{\mathcal C}_{kj} :=& -\frac{\im \nu \alpha^3 }{4}{\vec e}_k \cdot \int_B {\vec \xi} \times ({\vec \theta}_j + {\vec e}_j \times {\vec \xi} ) \dif {\vec \xi}, 
\\
{\mathcal N}_{kj} := &  \alpha^3 \left ( 1- \frac{\mu_0}{\mu_*} \right ) \int_B \left (
\delta_{kj} + \frac{1}{2}  {\vec e}_k \cdot  \nabla_\xi \times {\vec \theta}_j \right ) \dif {\vec \xi}. 
\end{align}\label{eqn:definerank2tensor}
\end{subequations}
These, in turn, rely on the vectoral solutions ${\vec \theta}_j$, $j=1,2,3$ to the transmission problem
\begin{subequations}
\begin{align}
\nabla_\xi \times \mu_*^{-1} \nabla_\xi \times {\vec \theta}_j - \im \omega \sigma_* \alpha^2 {\vec \theta }_j & = \im \omega \sigma_* \alpha^2 {\vec e}_j \times {\vec \xi}  && \text{in $B$ } ,\\
\nabla_\xi \cdot {\vec \theta}_j  = 0 , \qquad \nabla_\xi \times \mu_0^{-1} \nabla_\xi \times {\vec \theta}_j  & = {\vec 0} && \text{in ${\mathbb R}^3 \setminus B$ }, \\
[{\vec n} \times {\vec \theta}_j ]_\Gamma = {\vec 0},  \qquad [{\vec n} \times \mu^{-1} \nabla_\xi \times {\vec \theta}_j ]_\Gamma & = -2 [\mu^{-1 } ]_\Gamma {\vec n} \times {\vec e}_j  && \text{on $\Gamma:= \partial B$},\\
\int_\Gamma {\vec n} \cdot {\vec \theta}_j |_+ \dif {\vec \xi} & = 0, \\
{\vec \theta}_j & = O( | {\vec \xi} |^{-1}) && \text{as $|{\vec \xi} | \to \infty$ },
\end{align}\label{eqn:transproblemthetar2}
\end{subequations}
where $[\cdot ]_\Gamma $ denotes the  jump of the function  over $\Gamma$.
Note that ${\vec \theta}_j \ne ({\vec \theta})_j$, the latter being the $j$th component of a vector. Instead, the subscript $j$ should be interpreted as the $j$th solution of the transmission problem corresponding to the source terms in $B$ and on $\Gamma$ being constructed from the $j$th unit vector ${\vec e}_j$. In the above, we have dropped the subscript $\alpha$ on the position dependent $\mu$ as this problem is formulated for the object $B$ rather than $B_\alpha$.

\section{Complete asymptotic expansion of $ ({\vec H}_\alpha  - {\vec H}_0) ({\vec x} )$ as $\alpha \to 0$ } \label{sect:mainresult}

Our main result is
\begin{theorem}~\label{thm:main}
The magnetic field perturbation in the presence of a small conducting object $B_\alpha = \alpha B +{\vec z}$ for the eddy current model when $\nu$ is order one and ${\vec x}$ is away from the location ${\vec z}$ of the inclusion is completely described by the asymptotic formula
\begin{align}
({\vec H}_\alpha - {\vec H}_0) ({\vec x} )_i =&  \sum_{m=0}^{M-1} \sum_{p=0}^{M-1-m} 
( {\vec D}_{ x}^{2+m} G( {\vec x} , {\vec z} ))_{[i, {K}(m+1)]} \widecheck{\widecheck{\mathfrak M}}_{{K}(m+1) {J} (p+1)} ({\vec D}_z^p ({\vec H}_0  ( {\vec z} )))_{{J}(p+1)} +\nonumber \\
& ({\vec R}({\vec x}))_i, \label{eqn:mainresult} \\
{J}(p+1):=  &[j ,J(p)] = [j ,j_1,j_2,\cdots, j_p] , \nonumber \\
 {K}(m+1) :=&[k ,K(m)] = [k ,k_1,k_2,\cdots, k_m] ,\nonumber
\end{align}
with $|{\vec R}({\vec x})| \le C \alpha^{3+M}\| {\vec H}_0 \|_{W^{M+1,\infty}(B_\alpha)}$. In the above, $J(p)$ and $K(m)$ are $p$-- and $m$--tuples of integers, respectively, with each index taking values $1,2,3$.  Also
\begin{align}
 ( {\vec D}_{ x}^{2+m} G( {\vec x} , {\vec z} ))_{[i, {K}(m+1)]}  & =\left ( \prod_{\ell=1}^m  \partial_{x_{k_\ell}} \right ) (  \partial_{x_{k}} ( \partial_{x_{i} }( G( {\vec x} , {\vec z} )))) ,\nonumber \\
 ({\vec D}_z^p ({\vec H}_0  ( {\vec z} )))_{{J}(p+1)} & = \left (  \prod_{\ell=1}^p \partial_{z_{j_\ell}} \right )( {\vec H}_0({\vec z}) \cdot {\vec e}_j), \nonumber
\end{align}
and the coefficients of a rank $2+p+m$ generalised magnetic polarizability tensor (GMPT) are defined by
\begin{align}
\widecheck{\widecheck{\mathfrak M}}_{{K}(m+1) {J}(p+1)}  :=& -\widecheck{\mathfrak C}_{{K} (m+1) {J} (p+1) } + {\mathfrak N}_{ {K} (m+1) {J} (p+1)} , 
\end{align}
where
\begin{align}
\widecheck{\mathfrak C}_{ {K} (m+1) {J} (p+1) } :=& - \frac{\im \nu \alpha^{3+m+p}(-1)^m}{2(m+1)! p! (p+2)} {\vec e}_k \cdot  \nonumber \\
&\int_B {\vec \xi} \times \left (  ( \Pi( {\vec \xi}))_{K(m)} ({\vec \theta}_{{J} (p+1) } + (\Pi({\vec \xi}))_{J(p)} {\vec e}_j \times {\vec \xi} )\right ) \dif {\vec \xi} , \nonumber \\
{\mathfrak N}_{{K} (m+1) {J} (p+1) } : = & \left ( 1 - \frac{\mu_0}{\mu_*} \right ) \frac{ \alpha^{3+m+p} (-1)^m}{p! m!} {\vec e}_k \cdot \nonumber \\
& \int_B (\Pi( {\vec \xi}))_{K(m)}  \left ( \frac{1}{p+2} \nabla_\xi \times {\vec \theta}_{{J}(p+1) } +(\Pi( {\vec \xi}))_{J(p)} {\vec e}_j  \right ) \dif {\vec \xi}. \nonumber 
\end{align}
In the above, $ {\vec \theta}_{{J} (p+1)} $ satisfy the transmission problem
\begin{align*}
\nabla_\xi \times \mu_*^{-1} \nabla_\xi \times {\vec \theta}_{ {J}(p+1) } - \im \omega \sigma_* \alpha^2 {\vec \theta }_{ {J}(p+1) } & = \im \omega \sigma_* \alpha^2(\Pi({\vec \xi}))_{J(p)} {\vec e}_j \times {\vec \xi}  && \text{in $B$ }, \\
\nabla_\xi \cdot {\vec \theta}_{ {J} (p+1) }  = 0 , \qquad \nabla_\xi \times \mu_0^{-1} \nabla_\xi \times {\vec \theta}_{ {J}(p+1) }  & = {\vec 0} && \text{in ${\mathbb R}^3 \setminus B$ } , \\
[{\vec n} \times {\vec \theta}_{ {J} (p+1) }  ]_\Gamma &  = {\vec 0} && \text{on $\Gamma$},\\
  [{\vec n} \times \mu^{-1} \nabla_\xi \times {\vec \theta}_{{J} (p+1)}  ]_\Gamma & = -(p+2) [\mu^{-1 } ]_\Gamma ({\vec n} \times {\vec e}_j (\Pi ({\vec \xi}))_{J(p)} ) && \text{on $\Gamma$},\\
\int_\Gamma {\vec n} \cdot {\vec \theta}_{ {J} (p+1) } \dif {\vec \xi} & = 0 ,\\
{\vec \theta}_{ {J} (p+1) } & = O( | {\vec \xi} |^{-1}) && \text{as $|{\vec \xi} | \to \infty$ }, 
\end{align*}
$\displaystyle(\Pi({\vec \xi}))_{J(p)} := \prod_{\ell=1}^p  \xi_{j_\ell}= \xi_{j_1} \xi_{j_2} \cdots \xi_{j_p}$ and in the case $J(p)=\emptyset$ then $(\Pi({\vec \xi}))_{J(p)}=1$.
\end{theorem}

\begin{proof}
The expansion follows from the asymptotic formula in Theorem~\ref{thm:asympexpand}, and the results in Lemmas~\ref{lemma:reduction1}  and~\ref{lemma:reduction2} by noting that the coefficients of the rank $4+m+p$  tensor ${\mathfrak A}$ can be expressed in terms of  the coefficients of a rank $3+m+p$ and then a rank $2+m+p$  tensor by using the skew symmetry of their coefficients. Explicitly, we find that
{
\begin{align}
 {\mathfrak A}_{[[i,\ell, k,K(m)],J(p+1)]} = &  \varepsilon_{i k r} {\mathfrak C}_{[[r, \ell,K(m)],J(p+1)]} 
  = \varepsilon_{ \ell r s} \varepsilon_{i k r}  \widecheck{\mathfrak C}_{[[s,K(m) ],J(p+1)]} \nonumber \\
  = &  ( \delta_{\ell k} \delta_{s i} - \delta_{\ell i} \delta_{s k} )  \widecheck{\mathfrak C}_{[[s,K(m) ],J(p+1)]} =  \delta_{\ell k } \widecheck{\mathfrak C}_{[[i,K(m) ],J(p+1)]} - \delta_{\ell i} \widecheck{\mathfrak C}_{[[k,K(m) ],J(p+1)]} \nonumber \\
 = &  \delta_{\ell k } \widecheck{\mathfrak C}_{[[i,K(m) ],J(p+1)]} - \delta_{\ell i} \widecheck{\mathfrak C}_{K(m+1) J(p+1)} \nonumber ,
\end{align}}
where  $\varepsilon$ is as defined in (\ref{eqn:altensor}) and we have used $\varepsilon_{ \ell r s} \varepsilon_{i k r} =- \varepsilon_{ r\ell  s} \varepsilon_{r i k }=  \delta_{\ell k} \delta_{s i} - \delta_{\ell i} \delta_{s k} $.
Finally, by using 
\begin{equation}
\delta_{\ell  k} ({\vec D}_{ x}^{2+m} G({\vec x},{\vec z})  )_{[\ell ,{K}(m+1)]} =( {\vec D}_{ x}^{2+m} G({\vec x},{\vec z})  )_{[k,k,{K}(m)]} = ({\vec D}_{ x}^m ( {\vec D}_{ x}^2  G({\vec x},{\vec z}) )))_{[k,k,K(m)]} = 0,
\nonumber
\end{equation}
 since $ ({\vec D}_{ x}^2  G({\vec x},{\vec z}) )_{kk} = \hbox{trace} ({\vec D}_{ x}^2  G({\vec x},{\vec z}) )=0$, and by a term by term application of the above arguments, (\ref{eqn:mainresult}) is obtained.
\end{proof}

\begin{remark}
{Theorem~\ref{thm:main} extends the asymptotic expansion obtained by~\cite{ammarivolkov2013}, which provides the leading order term for $({\vec H}_\alpha - {\vec H}_0) ({\vec x} )$ as $\alpha \to 0$. We have previously shown in~\cite{ledgerlionheart2014} that this leading order term can be written in
the alternative form presented  in (\ref{eqn:rank2asymform}). In this case, $B$, $\alpha$, $\sigma_*$ and  $\mu_r=\mu_*/\mu_0$ are described by a complex symmetric rank 2 MPT $\widecheck{\widecheck{\mathcal M}}$, which is also a function of $\omega$. However, this description can only provide limited amounts of information about an object. Our new result reduces to this case when $M=1$.} {For $M=2$, our new result gives
\begin{align}
({\vec H}_\alpha - {\vec H}_0) ({\vec x} )_i =&  ({\vec D}_{ x}^2 G( {\vec x} , {\vec z}) )_{ik} \widecheck{\widecheck{\mathcal M}}_{kj} ( {\vec H}_0  ( {\vec z} ))_j +
({\vec D}_{ x}^2 G( {\vec x} , {\vec z} ))_{ik} \widecheck{\widecheck{\mathfrak M}}_{[k,[j, j_1]]} ({\vec D}_z( {\vec H}_0  ( {\vec z} )))_{[j,j_1]}\nonumber \\
& +( {\vec D}_{ x}^3 G( {\vec x} , {\vec z} ))_{[i,[k,k_1]]} \widecheck{\widecheck{\mathfrak M}}_{[[k,k_1],j]} ({\vec H}_0  ( {\vec z} ))_{j}+({\vec R}({\vec x}))_i , \label{eqn:rank3asymform}
\end{align}
with $|{\vec R}({\vec x})| \le C \alpha^{5}\| {\vec H}_0 \|_{W^{3,\infty}(B_\alpha)}$. In the above, $\widecheck{\widecheck{\mathfrak M}}_{[k,[j, j_1]]}   = -\widecheck{\mathfrak C}_{[k,[j, j_1]] } + {\mathfrak N}_{ [k,[j, j_1]]}$ where
\begin{subequations}
\begin{align} 
\widecheck{\mathfrak C}_{ [k,[j, j_1]]} =& - \frac{\im \nu \alpha^{4} }{6  } {\vec e}_k \cdot 
\int_B {\vec \xi} \times \left (    {\vec \theta}_{ [j,j_1] } + ({\vec \xi})_{j_1} {\vec e}_j \times {\vec \xi} \right ) \dif {\vec \xi} , \\
{\mathfrak N}_{[k,[j, j_1]] }  = &\alpha^{4} \left ( 1 - \frac{\mu_0}{\mu_*} \right )    {\vec e}_k \cdot 
 \int_B   \left ( \frac{1}{3} \nabla_\xi \times {\vec \theta}_{[j,j_1]} +({\vec \xi})_{j_1} {\vec e}_j  \right ) \dif {\vec \xi}.  
 \end{align}\label{eqn:rank3case1}
 \end{subequations}
Similarly, $\widecheck{\widecheck{\mathfrak M}}_{[[k,k_1], j]}= -\widecheck{\mathfrak C}_{[[k,k_1], j] } + {\mathfrak N}_{ [[k,k_1], j]}$ where 
\begin{subequations}
\begin{align} 
\widecheck{\mathfrak C}_{ [[k,k_1], j] } =&  \frac{\im \nu \alpha^{4} }{8} {\vec e}_k \cdot 
\int_B (  {\vec \xi})_{k_1}  {\vec \xi} \times \left (  {\vec \theta}_{j } +  {\vec e}_j \times {\vec \xi} \right ) \dif {\vec \xi} ,  \\
{\mathfrak N}_{[[k,k_1], j] }  = & -\alpha^4\left ( 1 - \frac{\mu_0}{\mu_*} \right )   {\vec e}_k \cdot 
 \int_B ( {\vec \xi})_{k_1}  \left ( \frac{1}{2} \nabla_\xi \times {\vec \theta}_j + {\vec e}_j  \right ) \dif {\vec \xi},
\end{align}\label{eqn:rank3case2}
\end{subequations}
and these can be computed using the solution of (\ref{eqn:transproblemthetar2}) already found for the computation of $ \widecheck{\widecheck{\mathcal M}}_{kj}$. For the computation of  (\ref{eqn:rank3case1}) the solution of 
\begin{align*}
\nabla_\xi \times \mu_*^{-1} \nabla_\xi \times {\vec \theta}_{ [j,j_1]} - \im \omega \sigma_* \alpha^2 {\vec \theta }_{ [j,j_1]} & = \im \omega \sigma_* \alpha^2( {\vec \xi})_{j_1} {\vec e}_j \times {\vec \xi}  && \text{in $B$ }, \\
\nabla_\xi \cdot {\vec \theta}_{ [j,j_1] }  = 0 , \qquad \nabla_\xi \times \mu_0^{-1} \nabla_\xi \times {\vec \theta}_{ [j,j_1] }  & = {\vec 0} && \text{in ${\mathbb R}^3 \setminus B$ } , \\
[{\vec n} \times {\vec \theta}_{ [j,j_1] }  ]_\Gamma &  = {\vec 0} && \text{on $\Gamma$},\\
  [{\vec n} \times \mu^{-1} \nabla_\xi \times {\vec \theta}_{[j,j_1]}   ]_\Gamma & = -3 [\mu^{-1 } ]_\Gamma ({\vec n} \times {\vec e}_j ( {\vec \xi})_{ j_1 } ) && \text{on $\Gamma$},\\
\int_\Gamma {\vec n} \cdot {\vec \theta}_{ [j,j_1] } \dif {\vec \xi} & = 0 ,\\
{\vec \theta}_{ [j,j_1] } & = O( | {\vec \xi} |^{-1}) && \text{as $|{\vec \xi} | \to \infty$ }, 
\end{align*}
is also required.} In the case of $M \to \infty$, Theorem~\ref{thm:main} provides a complete description of the field perturbation  $({\vec H}_\alpha - {\vec H}_0) ({\vec x} )$ caused by the presence of a permeable conducting object as $\alpha \to 0$. The object's shape and material properties in our new result are described by  $\widecheck{\widecheck{\mathfrak M}} $, which are GMPTs of increasing rank up to a maximum of $  M+1$ and are again functions of $B$, $\alpha$, $\sigma_*$, $\mu_r$ and $\omega$. By applying Theorem~\ref{thm:main} with $M > 1$,  $({\vec H}_\alpha - {\vec H}_0) ({\vec x} )$ can be more accurately described by including more of the higher order terms.  Complete asymptotic field expansions for small objects and GPTs and have previously been obtained for the EIT problem, acoustics and elasticity~\cite{ammarikang2003,ammarikanghelm,ammarikangbook}. But, to the best of the authors' knowledge, this is the first time they have been obtained for a Maxwell problem. {Like in~\cite{ammarivolkov2013}, our analysis makes the assumption that $B$ has a smooth boundary. The extension of the analysis to non-smooth boundaries will form part of our future work. However, numerical evidence from computing ${\mathcal M}$ for objects with edges indicates that our results are also likely to hold for such objects~\cite{ammarivolkov2013b,ledgerlionheart2014,ledgerlionheart2016}.}
\end{remark}

\begin{remark} \label{remark:objectcharacterisation}
To be able to characterise an unknown conducting permeable object {from measurements of  $({\vec H}_\alpha - {\vec H}_0)({\vec x})$}, using Theorem~\ref{thm:main}, a range of alternative approaches are possible, which include adapting the algorithms described by Ammari and Kang~\cite{ammarikanglecturenotes} for the EIT problem or  using a statistical classifier~\cite{ammarivolkov2013b}. In the latter case, we assume that we have a set of possible candidate objects and we follow Ammari and Kang~\cite[pg. 80]{ammarikanglecturenotes} to put these in canonical form such that the description $B_\alpha = \alpha B + {\vec z}$, for each object, implies that the origin for ${\vec \xi}$ coincides with the object's centre of mass and that the determinant of the P\'oyla-Szeg\"o tensor associated with $B$ (i.e. $ {\mathcal T}(\mu_r)[B]$ for $\mu_r\ne 1$ and $ {\mathcal T}(0)[B]$ for $\mu_r= 1$), it is equal to 1~\cite[pg. 80]{ammarikanglecturenotes}. {In an off-line stage}, the coefficients of $\widecheck{\widecheck{\frak M}}$ are then computed numerically for these objects  for a range of frequencies $\omega$ by solving the transmission problem for  $ {\vec \theta}_{{J} (p+1)} $ {to form a dictionary}. {In an on--line stage}, the unknown object's position ${\vec z}$ can be found by rotating the candidate objects {in the dictionary} (and hence their $\widecheck{\widecheck{\mathcal M}}$) and determining the best statistical fit for ${\vec z}$ by using measurements of $({\vec H}_\alpha-{\vec H}_0)({\vec x})$ and~(\ref{eqn:rank2asymform}). To find the unknown object's size,  it may be necessary to {ensure the dictionary includes $\widecheck{\widecheck{\mathcal M}}$ computed} at very small $\omega$, or frequencies at the limit of the eddy current model, such that $\widecheck{\widecheck{\mathcal M}}$ reduces to a suitably parameterised P\'oyla-Szeg\"o tensor for the candidate objects~\cite{ledgerlionheart2016}, and then repeat the above process to find the best fit for $\alpha$. To determine further geometrical and material parameter information, measurements of $({\vec H}_\alpha-{\vec H}_0)({\vec x})$ will be compared against Theorem~\ref{thm:main}  by using the known ${\vec z}$ and rotating the candidate objects {in the dictionary} (and hence their  $\widecheck{\widecheck{\frak M}}$) to find the best fit.
\end{remark}

\begin{remark}
{Currently, in practical magnetic induction metal detection, rather than solving (\ref{eqn:eddymodel}) with $\alpha=0$ to obtain the true background magnetic field, engineers frequently approximate the field 
 at a position ${\vec x}$ obtained from a small coil centred at ${\vec y}$ as that of a magnetic dipole}
\begin{equation}
{\vec H}_0({\vec x})_i  \approx {\vec D}^2 G({\vec x},{\vec y})_{ij}{\vec m}_j^e \label{eqn:dipoleapprox},
\end{equation}
where ${\vec m}_j^e$ is a constant vector that is a function of the coil's dimensions and the current flowing in it. However, {in walk through metal detectors, where there is little space between the coil arrays and the person being tested for a security threat, this does not provide an accurate representation of the field as the coils dimensions are no longer small compared to $|{\vec x}-{\vec y}|$ and the background field can vary considerably over the object. Engineers also assume that measurement coil, if treated as an emitter, will act as a dipole source. This means that for a single emitter--measurement coil arrangement the induced voltage
in a measurement coil located at ${\vec x}$ would be of the form of the leading order term for $({\vec H}_\alpha-{\vec H}_0)({\vec x})$~\cite{ledgerlionheart2016}
\begin{align}
({\vec m}^m)_i ({\vec H}_\alpha - {\vec H}_0) ({\vec x})_i \approx ({\vec m}^m)_i ( {\vec D}^2 G({\vec x},{\vec z}))_{ij} \widecheck{\widecheck{\mathcal M}}_{jk} ( {\vec D}^2 G({\vec z},{\vec y}))_{k\ell} ({\vec m}^e)_\ell . \label{eqn:rank2approx}
\end{align}
However, again the assumption of a dipole field for the measurement coil breaks down for walk through metal detectors and similar problems also exist in subsurface metal detection, when a metallic object is buried close to the surface. }

{Theorem~\ref{thm:main} can improve the  characterisation of hidden objects in magnetic induction metal detection by the following: Instead of (\ref{eqn:dipoleapprox})
the system  (\ref{eqn:eddymodel}) should be solved with $\alpha=0$ to obtain the true ${\vec H}_0({\vec x})$;  Rather than  (\ref{eqn:rank2approx})  $({\vec H}_\alpha - {\vec H}_0) ({\vec x} )$ should be integrated over an appropriate volume~\cite{ledgerlionheart2018} to obtain the correct induced voltage; Instead of just using the leading order term for $({\vec H}_\alpha-{\vec H}_0)({\vec x})$ as $\alpha \to 0$, and an object characterisation using $\widecheck{\widecheck{\mathcal M}}$, more
 terms in (\ref{eqn:mainresult}) should be used, and an object characterised by $\widecheck{\widecheck{\frak M}}$. Furthermore, object location and identification can then be improved by using  the approach described in Remark~\ref{remark:objectcharacterisation}. 
}
\end{remark}

\section{Derivation of the asymptotic formula} \label{sect:derive}

\subsection{Eliminating the current source}

We will build on~\cite{ammarivolkov2013,ledgerlionheart2014}, but, {in order to give a physical interpretation,} it is first instructive to rewrite the original problem described in Section~\ref{sect:problemconfig} with $ \alpha=0$ and $\alpha \ne0$ as transmission problems and to eliminate the current source. To do this, we note that in absence of an object, ${\vec H}_0 = \mu_0^{-1} \nabla \times {\vec A}_0$ and ${\vec A}_0$ solves
\begin{subequations} 
\begin{align}
\nabla \times \mu_0^{-1} \nabla \times {\vec A}_0 & = {\vec J}_0 && \text{in ${\mathbb R}^3$ } , \\
\nabla \cdot {\vec A}_0 & = 0 && \text{in ${\mathbb R}^3$},  \\
{\vec A}_0 & = O( | {\vec x} |^{-1}) && \text{as $|{\vec x} | \to \infty$ } .
\end{align}\label{eqn:transproblema0}
\end{subequations} 
Then, in the presence of the object, 
we can write  ${\vec H}_\alpha = \mu_\alpha^{-1} \nabla \times {\vec A}_\alpha$, ${\vec E}_\alpha = \im \omega {\vec A}_\alpha$ where, after appropriate gauging, ${\vec A}_\alpha$ solves
\begin{subequations} 
\begin{align}
\nabla \times \mu_*^{-1} \nabla \times {\vec A}_\alpha - \im \omega \sigma_* {\vec A}_\alpha & = {\vec 0} && \text{in $B_\alpha$ }, \\
\nabla \cdot {\vec A}_\alpha  = 0 , \ \nabla \times \mu_0^{-1} \nabla \times {\vec A}_\alpha  & = {\vec J}_0 && \text{in ${\mathbb R}^3 \setminus B_\alpha$ }, \\
[{\vec n} \times {\vec A}_\alpha]_{\Gamma_\alpha} = {\vec 0},  \ [{\vec n} \times \mu_\alpha^{-1} \nabla \times {\vec A}_\alpha ]_{\Gamma_\alpha} & = {\vec 0} && \text{on $\Gamma_\alpha := \partial B_\alpha$},\\
\int_{\Gamma_\alpha} {\vec n} \cdot {\vec A}_\alpha |_+ \dif {\vec x}  & = 0 , \\
{\vec A}_\alpha & = O( | {\vec x} |^{-1}) && \text{as $|{\vec x} | \to \infty$ }.
\end{align}\label{eqn:transproblemalpha}
\end{subequations}
Next, introducing ${\vec \xi} : = \displaystyle \frac{{\vec x} - {\vec z}}{\alpha}$, $ {\vec A}_\Delta ({\vec \xi}) : =\alpha  ({\vec A}_\alpha - {\vec A}_0)\left ( \displaystyle \frac{{\vec x}-{\vec z}}{\alpha} \right )  $ and rescaling we see that  $ {\vec A}_\Delta$  solves the following transmission problem
\begin{subequations}
\begin{align}
\nabla_\xi \times \mu_*^{-1} \nabla_\xi \times {\vec A}_\Delta - \im \omega \sigma_* \alpha^2 {\vec A}_\Delta & =  \im \omega \sigma_* \alpha   {\vec A}_0(  {\vec x} )  && \text{in $B$ } ,\\
\nabla_\xi \cdot {\vec A}_\Delta  = 0 , \qquad \nabla_\xi \times \mu_0^{-1} \nabla_\xi \times {\vec A}_\Delta  & = {\vec 0} && \text{in ${\mathbb R}^3 \setminus B$ } ,\\
[{\vec n} \times {\vec A}_\Delta]_\Gamma = {\vec 0},  \qquad [{\vec n} \times \mu^{-1} \nabla_\xi  \times {\vec A}_\Delta ]_\Gamma & = - [{\vec n} \times \mu^{-1}  \nabla_x \times {\vec A}_0( {\vec x}  ) ]_\Gamma && \text{on $\Gamma:= \partial B$},\\
\int_{\Gamma} {\vec n} \cdot {\vec A}_\Delta |_+  \dif {\vec \xi}  & = 0 , \\
{\vec A}_\Delta & = O( | {\vec \xi} |^{-1}) && \text{as $|{\vec \xi} | \to \infty$ },
\end{align} \label{eqn:transproblem}
\end{subequations}
where the current source no longer appears and, instead, is replaced by source terms in (\ref{eqn:transproblem}a) and (\ref{eqn:transproblem}c). {Electrical engineers would call ${\vec A}_\Delta ({\vec \xi})$ a scaled reduced vector potential.}
We now need to represent a polynomial vector field as the curl of another; we call this an {\em uncurling formula}.

\subsection{Uncurling formula}

\begin{lemma} \label{lemma:uncurl}
Given a smooth divergence free polynomial vector field in the form
\begin{align}
{\vec s}({\vec x}) = \sum_{p=0}^P \frac{1}{p!} ( {\vec D}_z^p ({\vec s})({\vec z}))_{ J(p+1)} (\Pi({\vec x}-{\vec z}))_{J(p) } {\vec e}_j   \label{eqn:curlform},
\end{align}
where $ \displaystyle( {\vec D}_z^p ( {\vec s} ) ( {\vec z} ) )_ { J(p+1)  } := \left  (\prod_{\ell=1}^p \partial_{z_{j_\ell}} \right )({\vec s} ({\vec z}) \cdot  {\vec e}_j ) =  \partial_{z_{j_1}} \partial_{z_{j_2}} \cdots \partial_{z_{j_p}} (
{\vec s}  ({\vec z})  \cdot  {\vec e}_j ) $ the field
\begin{align}
{\vec t}({\vec x}) = \sum_{p=0}^P \frac{1}{p!(p+2)} ( {\vec D}_z^p ({\vec s})({\vec z}))_{ J(p+1)} (\Pi({\vec x}-{\vec z}))_{J(p) } {\vec e}_j\times ({\vec x}- {\vec z})\label{eqn:uncurlform} ,
\end{align}
satisfies ${\vec s} = \nabla_x \times {\vec t}$.
\end{lemma}

\begin{proof}
We consider the $p$th term in (\ref{eqn:uncurlform}) and apply the standard identity $\nabla \times({\vec u} \times {\vec v}) = {\vec u} \nabla \cdot {\vec v} - {\vec v} \nabla \cdot {\vec u} + ({\vec v} \cdot \nabla) {\vec u} -  ({\vec u} \cdot \nabla ){\vec v}$ where the differentiation is with respect to ${\vec x}$ and
\begin{equation}
({\vec u})_j = \frac{1}{p!(p+2)} ({\vec D}_z^p ({\vec s}({\vec z})))_{{J}(p+1) }(\Pi({\vec x } - {\vec x} ))_{J(p)}, \qquad ({\vec v})_j = ({\vec x}-{\vec z})_j \nonumber .
\end{equation}
It is obvious that $\nabla \cdot {\vec v}=3$, $(\nabla {\vec v})_{ji} = \delta_{ji}$ and we can deduce 
\begin{align}
\nabla \cdot {\vec u} =  &  \frac{1}{p! (p+2) } \left (  ( {\vec D}_z^p ({\vec s}({\vec z})))_{ {J}(p+1) }  \left (  \delta_{j_1 j} (x_{j_2}-z_{j_2})\cdots  (x_{j_p}-z_{j_p})  + \cdots + \right . \right . \nonumber \\
& \left . \left  .  (x_{j_1}-z_{j_1})  (x_{j_2}-z_{j_2})\cdots  (x_{j_{p-1}}-z_{j_{p-1}})\delta_{j_p j} \right ) \right ) \nonumber \\
 =  &  \frac{1}{p!(p+2) }  \left ( ({\vec D}_z^p ({\vec s}({\vec z})))_{[ j, j , j_2, j_3,\cdots , j_p ]}    (x_{j_2}-z_{j_2}) (x_{j_3}-z_{j_3}) \cdots  (x_{j_p}-z_{j_p} ) + \cdots + \right .  \nonumber \\
 &    \left . ({\vec D}_z^p ({\vec s} ({\vec z})))_{[ j, j_1 , j_2, j_3,\cdots ,j_{p-1}, j ]}   (x_{j_1}-z_{j_1})   (x_{j_2}-z_{j_2}) \cdots  (x_{j_{p-1}}-z_{j_{p-1}}) \right )  = 0,  \nonumber
\end{align}
by interchanging the order of differentiation of ${\vec s}$ (e.g. $ ({\vec D}_z^p ({\vec s}( {\vec z})))_{[ j, j_1 , j_2, j_3,\cdots ,j_{p-1}, j ]} =\\ ({\vec D}_z^p ({\vec s} ({\vec z})))_{[ j,j, j_1 , j_2, j_3,\cdots ,j_{p-1} ]} $) and noting the repeated $j$ index, which imples $({\vec D}_z({\vec s} ({\vec z})))_{jj} = \text{tr} ( {\vec D}_z({\vec s} ({\vec z}))) = \nabla_z \cdot {\vec s}({\vec z}) =0$. Note also 
\begin{align}
(\nabla  {\vec u})_{ji} = &   \frac{1}{p!(p+2)}  \left ( ({\vec D}_z^p ({\vec s} ({\vec z})))_{{J}(p+1) }  \left (  \delta_{j_1 i} (x_{j_2}-z_{j_2})\cdots  (x_{j_p}-z_{j_p})  + \cdots + \right . \right .  \nonumber \\
& \left . \left . (x_{j_1}-z_{j_1})  (x_{j_2}-z_{j_2})\cdots  (x_{j_{p-1}}-z_{j_{p-1}})\delta_{j_p i} \right ) \right ) , \nonumber 
\end{align}
so that
\begin{align}
( ({\vec v} \cdot \nabla ){\vec u})_j = & (x_i-z_i)  \frac{1}{p!(p+2)} ({\vec D}_z^p ({\vec s}({\vec z})))_{{J}(p+1) }  \left (  \delta_{j_1 i} (x_{j_2}-z_{j_2})\cdots  (x_{j_p}-z_{j_p})  + \cdots + \right . \nonumber \\
& \left . (x_{j_1}-z_{j_1})  (x_{j_2}-z_{j_2})\cdots  (x_{j_{p-1}}-z_{j_{p-1}})\delta_{j_p i} \right ) \nonumber \\
= & \frac{p}{p!(p+2)} ({\vec D}_z^p ({\vec s} ({\vec z})))_{{J}(p+1) } (\Pi({\vec x } - {\vec z}) )_{J(p)} . \nonumber
\end{align}
Thus,
\begin{align}
(\nabla \times({\vec u} \times {\vec v}))_j =  & \frac{1}{p!} ({\vec D}_z^p ({\vec s} ({\vec z})))_{{J}(p+1) } (\Pi({\vec x } - {\vec z}) )_{J(p)}, \nonumber
\end{align}
and, by a term by term application of the above arguments, (\ref{eqn:curlform}) immediately follows.
\end{proof}

\begin{corollary}
An immediate consequence of Lemma~\ref{lemma:uncurl} and the smoothness of the divergence free $ \nabla \times {\vec E}_0 = \im \omega \mu_0 {\vec H}_0$ in $B_\alpha$ is that we can introduce a vector field ${\vec F} ({\vec x})$ as
\begin{align}
{\vec F} ({\vec x}): = &   \sum_{p=0}^P \frac{1}{p! ( p+2)} ({\vec D}_z^p (\nabla_z \times {\vec E}_0({\vec z})))_{{J}(p+1)} (\Pi({\vec x} - {\vec z}))_{J(p)} {\vec e}_j \times ( {\vec x} - {\vec z} ) ,\label{eqn:definef} 
\end{align}
where ${J}(p+1)=[j,J(p)]$ and whose curl is the polynomial vector field
\begin{align}
\nabla \times{\vec F} = &   \sum_{p=0}^P \frac{1}{p!} ({\vec D}_z^p (\nabla_z \times {\vec E}_0({\vec z})))_{ {J}(p+1)} (\Pi({\vec x} - {\vec z}))_{J(p)} {\vec e}_j  ,\label{eqn:definecurlf}
\end{align}
which is also the $P$th order Taylor series expansion of  $ \nabla \times {\vec E}_0 $ about ${\vec z}$ for $|{\vec x}-{\vec z}| \to 0$. Note that (\ref{eqn:definef}) and (\ref{eqn:definecurlf}) generalise the expressions for ${\vec F}({\vec x})$  and $ \nabla \times{\vec F}$ stated in~\cite{ammarivolkov2013}, which are for the case of $P=1$.
\end{corollary}

Furthermore, {a physical interpretation is helped} by writing ${\vec x} = \alpha {\vec \xi} + {\vec z}$ and constructing  
\begin{equation}
{\vec A}_0(\alpha {\vec \xi} + {\vec z}) = \mu_0 \sum_{p=0}^\infty \frac{\alpha^{1+p} }{p!(p+2)} ({\vec D}_z^p ({\vec H}_0({\vec z})))_{ {J}(p+1) } ( \Pi({\vec \xi }) )_{J(p)} {\vec e}_j \times {\vec \xi},  \label{eqn:expanda0full}
\end{equation}
in $B_\alpha$ such that
\begin{equation}
\nabla \times {\vec A}_0 (\alpha {\vec \xi} +{\vec z} ) = \mu_0  { \vec H}_0 (\alpha {\vec \xi} +{\vec z} ) = \mu_0 \sum_{p=0}^\infty \frac{\alpha^p }{p!} ({\vec D}_z^p ({\vec H}_0({\vec z})))_{ {J}(p+1)} (\Pi({\vec \xi})  )_{J(p)} {\vec e}_j \label{eqn:expandb0full} .
\end{equation}
\subsection{Accuracy of Taylor series approximations}
The smoothness of ${\vec H}_0= \frac{1}{\im \omega \mu_0} \nabla \times {\vec E}_0$  in $B_\alpha$ enables us to deduce that
\begin{equation}
\left \| \im \omega \mu_0 {\vec H}_0 ({\vec x})- \nabla \times {\vec F} 
 \right \|_{L^\infty ( B_\alpha)}
\le C \alpha^{1+P} \|  \nabla \times { \vec E}_0 \|_{W^{{P+1},\infty}(B_\alpha)}  \label{eqn:h0_curlflinf},
\end{equation}
where, here and throughout the following, the constant $C$ is independent of $\alpha$. Note that in the case of $P=1$  (\ref{eqn:h0_curlflinf}) is analogous to the bound (3.6) in~\cite{ammarivolkov2013}. It also follows from  (\ref{eqn:h0_curlflinf}) that
\begin{align}
\left \| \im \omega \mu_0{ \vec H}_0 ({\vec x})-  \nabla \times {\vec F}
 \right \|_{L^2 ( B_\alpha)}
\le &C \alpha^{\frac{3}{2}} \left \| \im \omega \mu_0 {\vec H}_0 ({\vec x})- \nabla \times {\vec F} \right \|_{L^\infty ( B_\alpha)}
\nonumber \\
\le & C \alpha^{\frac{5+2P}{2}} \| \nabla \times { \vec E}_0 \|_{W^{{P+1},\infty}(B_\alpha)} \label{eqn:curlh0_curlfl2}.
\end{align}

\subsection{Higher order energy estimates}

We follow the notation of ~\cite{ammarivolkov2013} and define
\begin{align}
{\vec X}_\alpha ({\mathbb R}^3) : =  &\left \{ {\vec u} : \frac{\vec u}{\sqrt{ 1+ |{\vec x}|^2 }}\in L^2({\mathbb R}^3 )^3 , \nabla \times {\vec u} \in  L^2({\mathbb R}^3 )^3, \nabla \cdot {\vec u} =0 \text{ in $B_\alpha^c$} \right \} ,\nonumber \\
\tilde{\vec X}_\alpha ({\mathbb R}^3) : = &  \left \{ {\vec u} : {\vec u} \in {\vec X}_\alpha ({\mathbb R}^3) , \ \int_{\Gamma_\alpha} {\vec u} \cdot {\vec n} |_+ \dif {\vec x} = 0 \right \} \nonumber ,
\end{align}
where $B_\alpha^c := {\mathbb R}^3 \setminus B_\alpha$. Using ${\vec E}_\alpha = \im \omega {\vec A}_\alpha$, the weak solution of (\ref{eqn:transproblemalpha}) can be written as: Find ${\vec E}_\alpha \in  \tilde{\vec X}_\alpha ({\mathbb R}^3) $ such that
\begin{align}
a_\alpha ({\vec E}_\alpha, {\vec v}) =\im \omega ({\vec J}_0,{\vec v})_{B_\alpha^c} \qquad \forall {\vec v}\in\tilde{\vec X}_\alpha ({\mathbb R}^3),
\end{align}
where
\begin{align}
a_\alpha ({\vec u},  {\vec v}) := ( \mu_\alpha^{-1} \nabla \times {\vec u}, \nabla \times {\vec v})_{{\mathbb R}^3} - \im \omega  (\sigma_\alpha {\vec u},{\vec v})_{B_\alpha} \nonumber,
\end{align}
and $(,)_D $ stands for the $L^2$ inner product on the domain $D\subseteq {\mathbb R}^3$. The weak solution of (\ref{eqn:transproblema0}) for ${\vec E}_0 = \im \omega {\vec A}_0$ is easily found and it can be shown that~\cite{ammarivolkov2013}
\begin{align}
 ( \mu_\alpha^{-1} \nabla& \times ( {\vec E}_\alpha - {\vec E}_0 ) ,  \nabla \times {\vec v})_{{\mathbb R}^3}  - \im \omega ( \sigma_\alpha ({\vec E}_\alpha -{\vec E}_0), {\vec v})_{B_\alpha}  = \nonumber \\
 &(\mu_0^{-1} - \mu_*^{-1} ) ( \nabla \times{\vec E}_0, \nabla \times {\vec v} )_{B_\alpha}  + \im \omega ( \sigma_\alpha {\vec E}_0, {\vec v})_{B_\alpha} \qquad \forall {\vec v}\in {\vec X}_\alpha ({\mathbb R}^3). \label{eqn:diffelphae0}
\end{align}
In a departure from~\cite{ammarivolkov2013}, we  define ${\vec w} \in \tilde{\vec X}_\alpha ( {\mathbb R}^3 )$ as the weak solution to
\begin{align}
a_\alpha ({\vec w}, {\vec v}) = &\im \omega \mu_0 ( \mu_0^{-1} - \mu_*^{-1} ) \left  ( {\sum_{p=0}^P \frac{1 }{p!} ({\vec D}_z^p ({\vec H}_0({\vec z})))_{{J}(p+1)} (\Pi ( {\vec x}- {\vec z})  )_{J(p)} {\vec e}_j , \nabla \times {\vec v}} \right )_{B_\alpha} \nonumber \\
&+ \im \omega ( \sigma_\alpha {\vec F},{\vec v})_{B_\alpha} \qquad
\forall {\vec v} \in \tilde{\vec X}_\alpha ({\mathbb R}^3), \label{eqn:bilineara}
\end{align}
and, if  we compare the above with their (3.7), we see it reduces to the latter for $P=1$ and also find that their
 Lemma 3.2 generalises to:

\begin{lemma}\label{lemma:revised32}
There exists a constant $C$  such that
\begin{align}
\left \| \nabla \times \left ( {\vec E}_\alpha - {\vec E}_0 -  {\vec w}  \right ) \right \|_{L^2(B_\alpha)} \le C &
 \left ( | 1 - \mu_r^{-1} | + \nu  \right ) \alpha^{\frac {5+2P}{2}} \| \nabla \times { \vec E}_0 \|_{W^{{P+1},\infty}(B_\alpha)} , \label{eqn:lemma:revised32a}\\ 
\left \|  {\vec E}_\alpha - {\vec E}_0 -\nabla \phi_0  - {\vec w}    \right \|_{L^2(B_\alpha)} \le C &
 \left ( | 1 - \mu_r^{-1} | + \nu  \right ) \alpha^ {\frac{7+2P}{2}} \| \nabla \times { \vec E}_0 \|_{W^{{P+1},\infty}(B_\alpha)}\label{eqn:lemma:revised32b} ,
\end{align}
where $\mu_r :=\mu_*/\mu_0$ and $\nu$ is as defined in (\ref{eqn:nudefine}).
\end{lemma}
\begin{proof}
The proof follows the steps in~\cite{ammarivolkov2013}, but  uses instead the higher order definitions of ${\vec F}$ and $\nabla \times {\vec F}$ stated in (\ref{eqn:definef}), (\ref{eqn:definecurlf}), respectively. The steps are the same until immediately before their (3.10). In our case, we have from (\ref{eqn:bilineara}) and (\ref{eqn:diffelphae0})
\begin{align}
( \mu_\alpha^{-1} & \nabla \times ( {\vec E}_\alpha - {\vec E}_0 - {\vec \Phi}_0 -{\vec w}), \nabla \times {\vec v} )_{{\mathbb R}^3}  - \im \omega ( \sigma_\alpha ( {\vec E}_\alpha - {\vec E}_0 - {\vec \Phi}_0 - {\vec w}) , {\vec v})_{B_\alpha} =\nonumber \\
&\im \omega \mu_0 ( \mu_0^{-1} - \mu_*^{-1} ) \left ( {\vec H}_0  - {\sum_{p=0}^P \frac{1 }{p!} ({\vec D}_z^p ({\vec H}_0({\vec z})))_{{J}(p+1)} (\Pi ( {\vec x}- {\vec z})  )_{J(p)} {\vec e}_j , \nabla \times {\vec v}} \right )_{B_\alpha} + \nonumber \\
&  \im \omega ( \sigma_\alpha ({\vec E}_0 + {\vec \Phi}_0 - {\vec F}),{\vec v})_{B_\alpha} \qquad
\forall {\vec v} \in \tilde{\vec X}_\alpha ({\mathbb R}^3), 
\end{align}
where ${\vec \Phi}_0 = \nabla \phi_0$ in $B_\alpha$ and ${\vec \Phi}_0 = \nabla  \tilde{\phi}_0$ in  $B_\alpha^c$ as defined in ~\cite{ammarivolkov2013}. Their (3.10) then becomes
\begin{align}
(\mu_0 \mu_\alpha^{-1} & \nabla \times ( {\vec E}_\alpha - {\vec E}_0 - {\vec \Phi}_0 -{\vec w}), \nabla \times {\vec v} )_{{\mathbb R}^3}  - \im k (  ( {\vec E}_\alpha - {\vec E}_0 - {\vec \Phi}_0 - {\vec w}) , {\vec v})_{B_\alpha} =\nonumber \\
&\im \omega \mu_0 (1 - \mu_r^{-1} ) \left ( {\vec H}_0  - {\sum_{p=0}^P \frac{1 }{p!} ({\vec D}_z^p ({\vec H}_0({\vec z})))_{{J}(p+1)} (\Pi ( {\vec x}- {\vec z})  )_{J(p)} {\vec e}_j , \nabla \times {\vec v}} \right )_{B_\alpha} + \nonumber \\
&  \im k ( ({\vec E}_0 + {\vec \Phi}_0 - {\vec F}),{\vec v})_{B_\alpha} ,\label{eqn:new310}
\end{align}
and we find from the Cauchy-Schwartz inequality, (\ref{eqn:definecurlf}) and (\ref{eqn:curlh0_curlfl2})  that
\begin{align}
&\left | \im \omega \mu_0\left ( {\vec H}_0  - {\sum_{p=0}^P \frac{1 }{p!} ({\vec D}_z^p ({\vec H}_0({\vec z})))_{{J}(p+1)} (\Pi ( {\vec x}- {\vec z})  )_{J(p)} {\vec e}_j , \nabla \times {\vec v}} \right )_{B_\alpha} \right | \le\nonumber \\
 &\qquad \qquad C \alpha^{\frac{5+2P}{2}}  \| \nabla \times { \vec E}_0 \|_{W^{{P+1},\infty}(B_\alpha)} \|  \nabla \times {\vec v} \|_{L^2(B_\alpha)}. \label{eqn:new310a}
\end{align}
Choosing ${\vec v} =  {\vec E}_\alpha - {\vec E}_0 - {\vec \Phi}_0 -{\vec w}$ in (\ref{eqn:new310}) and using (\ref{eqn:new310a}) then leads to the bound
\begin{align}
\| \nabla  \times ( {\vec E}_\alpha - {\vec E}_0  -{\vec w}) \|_{L^2(B_\alpha) } ^2 &\le C \alpha^{\frac{5+2P}{2}} | 1- \mu_r^{-1} |   \| \nabla \times { \vec E}_0 \|_{W^{{P+1},\infty}(B_\alpha)} \|  \nabla \times {\vec v} \|_{L^2(B_\alpha)} +\nonumber \\   
& k\| {\vec E}_0 + {\vec \Phi}_0 - {\vec F}\|_{L^2(B_\alpha)}  \|   {\vec v} \|_{L^2(B_\alpha)} \nonumber \\
\le &C \alpha^{\frac{5+2P}{2}} ( | 1- \mu_r^{-1} |    + \nu )  \| \nabla \times { \vec E}_0 \|_{W^{{P+1},\infty}(B_\alpha)} \|  \nabla \times {\vec v} \|_{L^2(B_\alpha)}  \label{eqn:310b},
\end{align}
where, in the last step, we have used $k=\nu / \alpha^2$,
\begin{align}
\|  {\vec E}_0 + {\vec \Phi}_0 - {\vec F}\|_{L^2(B_\alpha)}  \le&  C  \alpha  \|\nabla \times( {\vec E}_0  - {\vec F}) \|_{L^2(B_\alpha)} \le C \alpha^\frac{7+2P}{2} \| \nabla \times { \vec E}_0 \|_{W^{{P+1},\infty}(B_\alpha)}   \label{eqn:l2e0_f_gradphi} ,
\end{align}
and $  \|{\vec v}\|_{L^2(B_\alpha)}  \le C \alpha \| \nabla \times {\vec v}\|_{L^2(B_\alpha)} $.
 The result in (\ref{eqn:lemma:revised32a}) follows immediately from (\ref{eqn:310b}), and (\ref{eqn:lemma:revised32b}) follows from additionally using 
 \begin{align}
 \| {\vec E}_\alpha - {\vec E}_0 -\nabla \phi_0 - {\vec w} \|_{L^2(B_\alpha)} \le C \alpha  \| \nabla \times ( {\vec E}_\alpha - {\vec E}_0  - {\vec w}) \|_{L^2(B_\alpha)} \nonumber,
\end{align}
as obtained in ~\cite{ammarivolkov2013}.
\end{proof}

Using this result and ${\vec w} ({\vec x}) = \alpha {\vec w}_0 \left ( \frac{{\vec x}- {\vec z}}{\alpha} \right )$
we find that
  Theorem 3.1 in~\cite{ammarivolkov2013} immediately generalises to:
\begin{theorem} \label{thm:revised31}
There exists a constant $C$ 
  such that
\begin{align}
\left \| \nabla \times \left ( {\vec E}_\alpha - {\vec E}_0 - \alpha {\vec w}_0 \left ( \frac{{\vec x} - {\vec z}}{\alpha} \right ) \right ) \right \|_{L^2(B_\alpha)} \le C &
 \left ( | 1 - \mu_r^{-1} | + \nu  \right ) \alpha^{\frac{5+2P}{2}} \| \nabla \times { \vec E}_0 \|_{W^{{P+1},\infty}(B_\alpha)} , \\
\left \|  {\vec E}_\alpha - {\vec E}_0 -\nabla \phi_0 - \alpha {\vec w}_0 \left ( \frac{{\vec x} - {\vec z}}{\alpha} \right )  \right \|_{L^2(B_\alpha)} \le C &
 \left ( | 1 - \mu_r^{-1} | + \nu  \right ) \alpha^{\frac{7+2P}{2}} \| \nabla \times { \vec E}_0 \|_{W^{{P+1},\infty}(B_\alpha)} ,
\end{align}
\end{theorem}
Repeating their steps for the higher order terms we find that
\begin{equation}
{\vec w}_0 ({\vec \xi})  = \im \omega \mu_0  \sum_{p=0}^P  \frac{\alpha^p}{p!(p+2)}   ({\vec D}_z^p ({\vec H}_0({\vec z})))_{{J}(p+1)} {\vec \theta}_{{J}(p+1)} ({\vec \xi}) ,
\label{eqn:w0expand}
\end{equation}
where 
\begin{subequations}
\begin{align}
\nabla_\xi \times \mu_*^{-1} \nabla_\xi \times {\vec \theta}_{ {J}(p+1)} - \im \omega \sigma_* \alpha^2 {\vec \theta}_{{J}(p+1)}  = &   \im \omega \sigma_* \alpha^{2} (\Pi ({\vec \xi }) )_{J(p)}  {\vec e}_j \times {\vec \xi}  && \text{in $B$ } , \\
\nabla_\xi \cdot {\vec \theta}_{ {J}(p+1)}  =  0, \ \nabla_\xi \times \mu_0^{-1} \nabla_\xi \times {\vec \theta}_{ {J}(p+1)}   = & {\vec 0}  && \text{in ${\mathbb R}^3 \setminus B$ } , \\
[{\vec n} \times {\vec \theta}_{ {J}(p+1)} ]_\Gamma =  & {\vec 0},  && \text{on $\Gamma$} , \\
 [{\vec n} \times \mu^{-1} \nabla_\xi  \times {\vec \theta}_ { {J}(p+1)} ]_\Gamma  = &  - (p+2)[ \mu^{-1} ]_\Gamma  {\vec n} \times {\vec e}_j  ( \Pi(  {\vec \xi})  )_{J(p)} 
 && \text{on $\Gamma$}, \\
 \int_{\Gamma} {\vec n} \cdot {\vec \theta}_{ {J}(p+1)}  |_+ \dif {\vec \xi}   = & 0 , \\
{\vec \theta}_{ {J}(p+1)}  = & O( | {\vec \xi} |^{-1}) && \text{as $|{\vec \xi} | \to \infty$ }.
\end{align} \label{eqn:transproblemtheta}
\end{subequations}

\begin{remark}
The indices on the solution  ${\vec \theta}_{{J}(p+1)} ({\vec \xi}) $ to the auxiliary problem (\ref{eqn:transproblemtheta}) should be interpreted differently to the tensoral indices previously presented. They should be interpreted as the vector valued solution when the source terms in $B$ and on $\Gamma$ contain the product ${\vec e}_j (\Pi({\vec \xi}))_{J(p)}$. The transmission problem for ${\vec \theta}_{{J}(p+1)} ({\vec \xi}) $ is independent of the object's position and is independent of the background excitation. It depends only on the shape of the object, its size, material properties and the frequency of the excitation. It generalises the transmission problem stated in (\ref{eqn:transproblemthetar2}), obtained in \cite{ammarivolkov2013}, and reduces to this case when $p=0$.  We will examine the transformation of ${\vec \theta}_{ {J}(p+1)} ({\vec \xi}) $ under rotations and/or reflections of the object in the proof of Lemma~\ref{lemma:tensors}.
\end{remark}

Alternatively, by using (\ref{eqn:expanda0full}) and (\ref{eqn:expandb0full}), and substituting in to the source terms in $B$ and on $\Gamma$ in (\ref{eqn:transproblem}), we see that it is possible to write 
\begin{equation}
{\vec A}_\Delta ({\vec \xi})  = \sum_{p=0}^\infty \mu_0 \frac{\alpha^p}{p!(p+2)}   ({\vec D}_z^p ({\vec H}_0({\vec z})))_{ {J}(p+1)} {\vec \theta}_{{J}(p+1)} ({\vec \xi}) , \label{eqn:Adelta}
 \end{equation}
 which, by truncating at $P$ terms and multiplying by $\im \omega $, coincides with the weak solution ${\vec w}_0 ({\vec \xi})$ and a hence provides a physical interpretation for the latter.

\subsection{Integral representation formula}
 Ammari {\it et al.}~\cite{ammarivolkov2013} have derived the following Stratton--Chu type formula
\begin{align*}
({\vec H}_\alpha  -{\vec H}_0)({\vec x}) =& \int_{B_\alpha} \nabla_x G({\vec x},{\vec y}) \times \nabla_y \times ({\vec H}_\alpha-{\vec H}_0)({\vec y})
 \dif {\vec y} +\nonumber \\
 & \left ( 1- \frac{\mu_*}{\mu_0}  \right ) \int_{B_\alpha} ( {\vec H}_\alpha({\vec y}) \cdot \nabla_y ) \nabla_x G({\vec x},{\vec y}) \dif {\vec y} ,
 \end{align*}
 for ${\vec x}$ exterior to $B_\alpha$,  which relates the magnetic field perturbation outside the object to the magnetic field in its interior. By introducing the representation for $B_\alpha$, and using the eddy current equations (\ref{eqn:eddymodel}), we have the alternative form
 \begin{align}
  ({\vec H}_\alpha  -{\vec H}_0)({\vec x}) =  &  \sigma_* \int_{B_\alpha} \nabla_x G({\vec x},{\vec y} )\times {\vec E}_\alpha ( {\vec y}  ) \dif {\vec y}-
  \left ( 1- \frac{\mu_*}{\mu_0}  \right ) \int_{B_\alpha} {\vec D}_{ x} ^2G({\vec x}, {\vec y }  )  {\vec H}_\alpha(  {\vec y} )   \dif {\vec y} \nonumber \\  
 = & \mathrm{I} + \mathrm{II} .\label{eqn:intrep}
 \end{align}

\subsection{Asymptotic formula}

One approach to approximating integrals in (\ref{eqn:intrep}) is to transform the domain of integration from $B_\alpha $ to $B$, to express ${\vec E}_\alpha (\alpha {\vec \xi}+ {\vec z})$ and ${\vec H}_\alpha (\alpha {\vec \xi}+ {\vec z})$ in terms of ${\vec A}_\Delta ( {\vec \xi})$ and ${\vec A}_0  (\alpha {\vec \xi}+ {\vec z}) $ (and their curls) and then to
substitute in truncated expansions of (\ref{eqn:expanda0full}) and (\ref{eqn:Adelta}). However, as rigorous estimates for these approximations are not available, we would not be able to quantify the remainder and so,
instead, we pursue  the previously presented approach in~\cite{ammarivolkov2013}, which uses  weak solutions and has a rigorous theoretical framework. We extend this approach to the higher order case  by using the bounds we have derived in Theorem~\ref{thm:revised31} and the result is the following theorem, which generalises their Theorem 3.2.

 \begin{theorem} \label{thm:asympexpand}
 Let $\nu$ be order one and let $\alpha$ be small. For ${\vec x}$ away from the location ${\vec z}$ of the inclusion we have
 \begin{align}
   ({\vec H}_\alpha&-{\vec H}_0)({\vec x})
   = - \im \nu \alpha^3 \sum_{m=0}^{M-1} \sum_{p=0}^{M-1-m} \frac{(-1)^m \alpha^{p+m}}{p!(m+1)! (p+2)}\nonumber\\
   & \int_B  (({\vec D}_x^{2+m} G({\vec x},{\vec z})  {\vec \xi} )_{{K}(m+1)} {\vec e}_k ( \Pi ({\vec \xi}))_{K(m)}) \times \nonumber \\
& \left ( ({\vec D}_z^p ({\vec H}_0({\vec z})))_{ {J}(p+1)}  (   {\vec \theta}_{{J}(p+1)} + (\Pi({\vec \xi}))_{J(p)} {\vec e}_j \times {\vec \xi} )  \right ) \dif {\vec \xi}+\nonumber \\
& \alpha^3 \left ( 1 - \frac{\mu_0}{\mu_*} \right )  \sum_{m=0}^{M-1} \sum_{p=0}^{M-1-m}
\frac{(-1)^m  \alpha^{p+m} }{p! m!} ( {\vec D}_x^{2+m} G ({\vec x},{\vec z}))_{[i,K(m+1)]} ({\vec e}_i \otimes {\vec e}_k) \nonumber \\
& \int_B (\Pi( {\vec \xi}))_{K(m)}  ({\vec D}_z^p ({\vec H}_0({\vec z})))_{ {J}(p+1)}  \left ( \frac{1}{p+2} \nabla_\xi \times {\vec \theta}_{{J}(p+1) } +( \Pi( {\vec \xi}))_{J(p)} {\vec e}_j  \right ) \dif {\vec \xi}   + {\vec R}({\vec x}),
 \end{align}
 where $|{\vec R}({\vec x})| \le C\alpha^{3+M}   \| {\vec H}_0 \|_{W^{{M+1}, \infty} (B_\alpha)}  $ .
 \end{theorem}
 \begin{proof}
 The result follows immediately from Lemmas~\ref{lemma:boundI} and~\ref{lemma:boundII} presented in Section~\ref{sect:proofasymp}.
 \end{proof}
 
\begin{corollary}
It follows from Lemma~\ref{lemma:tensors} that an alternative form of Theorem~\ref{thm:asympexpand} is
 \begin{align}
   ({\vec H}_\alpha & -{\vec H}_0)({\vec x})_i =  \sum_{m=0}^{M-1} \sum_{p=0}^{M-1-m}  ({\vec D}_x^{2+m} G({\vec x},{\vec z})  )_{[\ell ,{K}(m+1)]} { {\mathfrak A}_{[[i,\ell, {K}(m+1)],{J}(p+1)]}}
    ( {\vec D}_z^p({\vec H}({\vec z})))_{ {J}(p+1)} \nonumber \\
   & + \sum_{m=0}^{M-1} \sum_{p=0}^{M-1-m}  ({\vec D}_x^{2+m} G({\vec x},{\vec z})  )_{[i,{K}(m+1)]}  {\mathfrak N}_{{K}(m+1) {J}(p+1) } ( {\vec D}_z^p({\vec H}({\vec z})))_{ {J}(p+1)} +  ({\vec R}({\vec x}))_i ,\label{eqn:fulltensorexp}
  \end{align}
  where $|{\vec R}({\vec x})| \le C\alpha^{3+M}   \| {\vec H}_0 \|_{W^{{M+1}, \infty} (B_\alpha)}  $ and
  \begin{align}
 {  {\mathfrak A}_{[[i,\ell, {K}(m+1)],{J}(p+1)]}} : = & - \im \nu \frac{(-1)^m \alpha^{3+p+m}}{p!(m+1)! (p+2)}  {\vec e}_i \cdot \nonumber \\
&    \int_B   {\vec e}_k \times 
( ( {\vec \xi} )_{\ell }  (\Pi({\vec \xi}))_{K(m)}    (  {\vec \theta}_{{J}(p+1)}  + (\Pi({\vec \xi}))_{J(p)} {\vec e}_j \times {\vec \xi} ) )  \dif {\vec \xi} ,\label{eqn:Arank4h} \\
 {\mathfrak N}_{{K}(m+1) {J}(p+1) }:= & \frac{ (-1)^m\alpha^{3+p+m} }{p! m!} \left ( 1 - \frac{\mu_0}{\mu_*} \right ) {\vec e}_k \cdot \nonumber \\
 &  \int_B ( \Pi ( {\vec \xi} ))_{K(m)}   \left ( \frac{1}{p+2} \nabla_\xi \times {\vec \theta}_{{J}(p+1) } +(\Pi( {\vec \xi}))_{J(p)} {\vec e}_j  \right ) \dif {\vec \xi} , \label{eqn:Nrank2h}
\end{align} 
are the coefficients of rank $4+m+p$ and $2+m+p$ tensors, respectively.
\end{corollary}

\subsection{Results for the proof of the asymptotic formula} \label{sect:proofasymp}

It is useful to note for ${\vec x} $ away from $B_\alpha$ and ${\vec y}$ in $B_\alpha$ that $ G({\vec x}, {\vec y}) $ is smooth and analytic where we plan to use it and so the Taylor series expansions
\begin{align}
\nabla_x G({\vec x},  {\vec y}) = & \sum_{m=0}^\infty \frac{(-1)^m }{m!} ( {\vec D}_x^m ( \nabla_x G({\vec x},{\vec z})))_{K(m)} (\Pi(  {\vec y}-{\vec z}) )_{K(m)} ,\label{eqn:taylorgradg} \\
{\vec D}_x^2 G({\vec x},  {\vec y}) =  &\sum_{m=0}^\infty  \frac{(-1)^m  }{m!} ( {\vec D}_x^m ( {\vec D}_x^2 G ({\vec x},{\vec z})))_{K(m)} ( \Pi( {\vec y}-{\vec z}) )_{K(m)}  ,\label{eqn:taylord2g}
\end{align}
converge as $|{\vec y}-{\vec z}| \to 0$. Consequently, we have the estimates
\begin{align}
 & \left \| \nabla_x G({\vec x}, {\vec y} )- \sum_{m=0}^Q \frac{(-1)^m  }{m!} ( {\vec D}_x^m ( \nabla_x G({\vec x},{\vec z})))_{K(m)} ( \Pi( {\vec y} -{\vec z} ))_{K(m)}  \right \|_{L^2(B_\alpha)} \le C \alpha^{\frac{5+2Q}{2}} ,  \label{eqn:est3} \\
 & \left  \|  {\vec D}_x^2 G({\vec x},  {\vec y})
- \sum_{m=0}^S  \frac{(-1)^m }{m!} ( {\vec D}_x^m ( {\vec D}_x^2 G ({\vec x},{\vec z})))_{K(m)} ( \Pi ( {\vec y}-{\vec z} ) )_{K(m)}    \right \|_{L^2 (B_\alpha)} \le C \alpha^{\frac{5+2S}{2}} . \label{eqn:est4}
\end{align}

\subsubsection{Approximation of $\mathrm{I}$}
In similar way to ~\cite{ammarivolkov2013}, we write $\mathrm{I}=\mathrm{I}_1+\mathrm{I}_2+\mathrm{I}_3+\mathrm{I}_4$ where
\begin{align*}
\mathrm{I}_1 = & \sigma_* \int_{B_\alpha} \nabla_x G({\vec x}, {\vec y}) \times \left ( {\vec E}_\alpha ({\vec y}) -{\vec E}_0({\vec y}) -\nabla_y \phi_0 ({\vec y})- \alpha {\vec w}_0 \left ( \frac{ {\vec y}- {\vec z}}{\alpha } \right ) \right ) \dif {\vec y} , \\
\mathrm{I}_2 = & \sigma_* \int_{B_\alpha} \nabla_x G({\vec x},  {\vec y}) \times  \left ( {\vec E}_0 ({\vec y}) + \nabla_y \phi_0 ({\vec y}) - {\vec F}({\vec y}) \right ) \dif {\vec y}, \\
\mathrm{I}_3 = & \sigma_* \int_{B_\alpha} \left ( \nabla_x G ({\vec x},  {\vec y})  -
\sum_{m=0}^Q \frac{(-1)^m }{m!} ( {\vec D}_x^m ( \nabla_x G({\vec x},{\vec z})))_{K(m)} ( \Pi ( {\vec y}-{\vec z}))_{K(m)}  \right ) \times \nonumber \\
&\left ( {\vec F}({\vec y})+ \alpha{\vec w}_0 \left ( \frac{{\vec y} - {\vec z}}{\alpha} \right ) \right ) \dif {\vec y} , \\
\mathrm{I}_4 = & \sigma_* \int_{B_\alpha} \sum_{m=0}^Q \frac{(-1)^m }{m!} ( {\vec D}_x^m ( \nabla_x G({\vec x},{\vec z})))_{K(m)} ( \Pi ( {\vec y}-{\vec z}))_{K(m)} 
\times
\left ( {\vec F}({\vec y})+ \alpha{\vec w}_0 \left ( \frac{{\vec y} - {\vec z}}{\alpha} \right ) \right ) \dif {\vec y}.
\end{align*}

\begin{lemma} \label{lemma:boundi1i2i3}
We can bound $\mathrm{I}_1$, $\mathrm{I}_2$ and $\mathrm{I}_3$ as
\begin{align}
|\mathrm{I}_1| \le&  C \alpha^{3+P}   \|{\vec H}_0 \|_{W^{P+1,\infty}(B_\alpha)},  \qquad
|\mathrm{I}_2| \le  C \alpha^{3+P}   \|{\vec H}_0 \|_{W^{P+1,\infty}(B_\alpha)} , \nonumber\\
|\mathrm{I}_3| \le&  C \alpha^{3+Q}   \|{\vec H}_0 \|_{W^{P+1,\infty}(B_\alpha)} . \nonumber
\end{align}
\end{lemma}
\begin{proof}
Using Theorem~\ref{thm:revised31} we have
\begin{align}
|\mathrm{I}_1| \le &  C \alpha^{\frac{3}{2} } \sigma_* \left ( | 1 - \mu_r^{-1} | + \nu  \right ) \alpha^{\frac{7+2P}{2}} \| \nabla \times { \vec E}_0 \|_{W^{{P+1},\infty}(B_\alpha)} \nonumber\\
\le &  C k \left ( | 1 - \mu_r^{-1} | + \nu  \right ) \alpha^{ 5+P } \|   { \vec H}_0 \|_{W^{{P+1},\infty}(B_\alpha)} \nonumber,
\end{align}
where the second inequality follows from $\nabla \times {\vec E}_0 = \im \omega \mu_0 {\vec H}_0$ and (\ref{eqn:nudefine}). The final result for $\mathrm{I}_1$ follows by recalling $\mu_r =O(1)$ and $\nu =k \alpha^2= O(1)$. Next, using (\ref{eqn:l2e0_f_gradphi}) we find
\begin{align}
|\mathrm{I}_2| \le & C \alpha^{\frac{3}{2}}  \sigma_* \alpha^{\frac{7+2P}{2}  } \|  \nabla \times { \vec E}_0 \|_{W^{{P+1},\infty}(B_\alpha)} \le  C k \alpha^{ 5+P } \|   { \vec H}_0 \|_{W^{{P+1},\infty}(B_\alpha)} \nonumber,
\end{align}
and the result for $\mathrm{I}_2$ is then easily obtained. Finally, for $\mathrm{I}_3$, we have from (\ref{eqn:definef}), (\ref{eqn:w0expand}) and (\ref{eqn:est3})  that
\begin{align}
|\mathrm{I}_3| \le & C \alpha^{\frac{5+2Q}{2} } \alpha \alpha^{\frac{3}{2} } \sigma_*  \|  \nabla \times { \vec E}_0 \|_{W^{{P+1},\infty}(B_\alpha)}  \le  C k \alpha^{ 5+Q } \|   { \vec H}_0 \|_{W^{{P+1},\infty}(B_\alpha)} \nonumber,
\end{align}
and the result for $\mathrm{I}_3$ is then easily obtained.
\end{proof}


\begin{lemma} \label{lemma:m0I4}
The term corresponding to $m=0$ in $\mathrm{I}_4$ vanishes so that
\begin{align}
\mathrm{I}_4   = & \im \omega \alpha^4 \sigma_* \int_B  \sum_{m=1}^Q \frac{(-1)^m \alpha^m}{m!} ( {\vec D}_x^m ( \nabla_x G({\vec x},{\vec z})))_{K(m)} (\Pi ( {\vec \xi} ) )_{K(m)}  \times\nonumber \\
{} & \left ( \sum_{p=0}^P \mu_0 \frac{\alpha^p}{p!(p+2)}   ({\vec D}_z^p ({\vec H}_0({\vec z})))_{{J}(p+1)}(  {\vec \theta}_{{J}(p+1)}  + (\Pi ( {\vec \xi}))_{J(p)} {\vec e}_j \times {\vec \xi} )  \right ) \dif {\vec \xi}  . \nonumber
\end{align}
\end{lemma}
\begin{proof}
The term corresponding to $m=0$ in $\mathrm{I}_4$ is
\begin{align}
  \im \omega \alpha^4  \sigma_* & \int_B   \nabla_x G({\vec x},{\vec z} )   \times  \left ( \sum_{p=0}^P \mu_0 \frac{\alpha^p}{p!(p+2)}   ({\vec D}_z^p ({\vec H}_0({\vec z})))_{{J}(p+1)}(  {\vec \theta}_{{J}(p+1)} + (\Pi ( {\vec \xi}))_{J(p)} {\vec e}_j \times {\vec \xi} )  \right ) \dif {\vec \xi} \nonumber \\
  = &  \im \omega \sigma_* \mu_0 \alpha^4 \sum_{p=0}^P   \frac{\alpha^p}{p!(p+2) }   ({\vec D}_z^p ({\vec H}_0({\vec z})))_{{J}(p+1)}  
  \nabla_x G({\vec x},{\vec z} )   \times   \int_B (  {\vec \theta}_{ {J}(p+1) } 
    + (\Pi ( {\vec \xi}))_{J(p)} {\vec e}_j \times {\vec \xi} )  \dif {\vec \xi} . \label{eqn:mzero}
\end{align}
By  applying integration by parts and using the transmission problem (\ref{eqn:transproblemtheta}) we have
\begin{align}
\int_B &(  {\vec \theta}_{ \tilde{J}(p) } 
    + (\Pi( {\vec \xi}))_{J(p)} {\vec e}_j \times {\vec \xi} )  \dif {\vec \xi} = \frac{1}{\im \omega \sigma_* \alpha^2 } \int_B \nabla_\xi \times \mu_*^{-1} \nabla_\xi \times {\vec \theta}_{ {J}(p+1)} \dif  {\vec \xi} \nonumber \\
 &=    \frac{1}{\im \omega \sigma_* \alpha^2 } \int_{B\cup B^c} \nabla_\xi \times \mu^{-1} \nabla_\xi \times {\vec \theta}_{ {J}(p+1)} \dif  {\vec \xi} = \frac{1}{\im \omega \sigma_* \alpha^2 }\int_\Gamma [ \mu^{-1} \nabla_\xi \times  {\vec \theta}_{{J}(p+1)}  \times {\vec n}^+ ]_\Gamma \dif {\vec \xi} \nonumber \\
  &=    - \frac{(p+2) }{\im \omega \sigma_* \alpha^2  } [ \mu^{-1}  ]_\Gamma \int_B \nabla_\xi \times ( (\Pi(  {\vec \xi}))_{J(p)} {\vec e}_j )  \dif {\vec \xi} \label{eqn:intbypartsontheta},
   \end{align} 
   where $B^c:= {\mathbb R}^3 \setminus B$.   Using the alternating tensor $\varepsilon$, whose coefficients satisfy
\begin{align}
\varepsilon_{ijk} := \left \{ \begin{array}{rl} 1 & \text{if $(i,j,k)$ is a cyclic permutation of $(1,2,3)$} \\
-1 & \text{if $(i,j,k)$ is an anti-cyclic permutation of $(1,2,3)$} \\
0  & \text{if  any of $i,j,k$ are equal}  \end{array} \right . \label{eqn:altensor} ,
\end{align}
we find that 
\begin{align}
( \nabla_\xi \times & (  ({\vec D}_z^p ({\vec H}_0({\vec z})))_{{J}(p+1)} (\Pi (  {\vec \xi}))_{J(p)}{\vec e}_j  ))_i =  \varepsilon_{ikj}  ({\vec D}_z^p ({\vec H}_0({\vec z})))_{{J}(p+1)}  \frac{\partial}{\partial \xi_k}(  \Pi( {\vec \xi}))_{J(p)} ) \nonumber \\
& =\varepsilon_{ikj}  ({\vec D}_z^p ({\vec H}_0({\vec z})))_{{J}(p+1)} \left ( \delta_{k j_1} \xi_{j_2} \cdots \xi_{j_p} + \cdots + \xi_{j_1} \cdots \xi_{j_{p-1}} \delta_{k j_p} \right ) \nonumber \\
& = p \varepsilon_{ikj}   ({\vec D}_z^p ({\vec H}_0({\vec z})))_{ [ j, k, j_2, \cdots, j_p ]}  ( \xi_{j_2} \cdots \xi_{j_p}   ) =0 \label{eqn:curlXizero},
\end{align}
since $\varepsilon_{ikj} = -\varepsilon_{ijk}$ and $({\vec D}_z^p ({\vec H}_0({\vec z})))_{ [ j, k, j_2, \cdots , j_p]} = ({\vec D}_z^p ({\vec H}_0({\vec z})))_{ [ k, j, j_2, \cdots, j_p]} $. The desired result then immediately follows from (\ref{eqn:mzero}), (\ref{eqn:intbypartsontheta}) and  (\ref{eqn:curlXizero}).

\end{proof}


 \begin{lemma}  \label{lemma:boundI}
 The integral $\mathrm{I}$ can be expressed as 
 \begin{align}
\mathrm{I} = - \im \nu \alpha^3 \sum_{m=0}^{M-1} \sum_{p=0}^{M-1-m} \frac{(-1)^m \alpha^{p+m}}{p!(m+1)! (p+2)} \int_B  ({\vec D}_x^{2+m} G({\vec x},{\vec z})  {\vec \xi} )_{K(m+1)}  (\Pi({\vec \xi}))_{K(m)}{\vec e}_k  \times \nonumber \\
 \left ( ({\vec D}_z^p ({\vec H}_0({\vec z})))_{{J}(p+1)}  (   {\vec \theta}_{{J}(p+1)}  + (\Pi ( {\vec \xi}))_{J(p)} {\vec e}_j \times {\vec \xi} )  \right ) \dif {\vec \xi} + {\vec R}({\vec x}),
 \end{align}
 where $|{\vec R}({\vec x})| \le C\alpha^{3+M}   \| {\vec H}_0 \|_{W^{{M+1}, \infty} (B_\alpha)}  $  .
  \end{lemma}
 
 \begin{proof}
Recall that $\mathrm{I}=\mathrm{I}_1+\mathrm{I}_2+\mathrm{I}_3 +\mathrm{I}_4$ and choose $P=Q=M$. We then see from Lemma~\ref{lemma:boundi1i2i3} that $\mathrm{I}_1$, $\mathrm{I}_2$ and $\mathrm{I}_3$ all form part of ${\vec R}({\vec x})$. Using Lemma ~\ref{lemma:m0I4} we find that
\begin{align}
\mathrm{I}_4  = & \im \omega \alpha^4 \sigma_* \int_B  \sum_{m=1}^{M} \frac{(-1)^m \alpha^m}{m!} ( {\vec D}_x^m ( \nabla_x G({\vec x},{\vec z})))_{K(m)} (\Pi( {\vec \xi}))_{K(m)} \times \nonumber \\
{} & \left ( \sum_{p=0}^{M} \mu_0 \frac{\alpha^p}{p!(p+2)}   ({\vec D}_z^p ({\vec H}_0({\vec z})))_{{J}(p+1)}(  {\vec \theta}_{{J}(p+1)} +(\Pi (  {\vec \xi}))_{J(p)} {\vec e}_j \times {\vec \xi} )  \right ) \dif {\vec \xi}   \nonumber \\
= & - \im \nu \alpha^3  \sum_{m=0}^{M-1} \sum_{p=0}^{M-1-m}  \frac{(-1)^m \alpha^{m+p}}{(m+1)! p! (p+2)} 
\int_B  ({\vec D}_x^{m+2}  G({\vec x},{\vec z} ) {\vec \xi})_{{K}(m+1)} ( \Pi( {\vec \xi}))_{K(m)} {\vec e}_k \times \nonumber \\
&  \left ( ({\vec D}_z^p ({\vec H}_0({\vec z})))_{{J}(p+1)}(  {\vec \theta}_{{J}(p+1)} + (\Pi ({\vec \xi}))_{J(p)} {\vec e}_j \times {\vec \xi} )  \right ) \dif {\vec \xi}   +{\vec R}_{\mathrm{I}}({\vec x}) , \nonumber
\end{align}
where
\begin{align}
{\vec R}_{\mathrm{I}}({\vec x}) = & -  \im \nu \alpha^3 \sum_{m=0}^{M-1}  \sum_{p=M-m}^{M} \frac{(-1)^m \alpha^{m+p}}{(m+1)! p! (p+2)} 
\int_B  ({\vec D}_x^{m+2}   G({\vec x},{\vec z}  ) {\vec \xi})_{ {K}(m+1)} ( \Pi  ( {\vec \xi}))_{K(m)} {\vec e}_k \times \nonumber \\
&  \left ( ({\vec D}_z^{p} ({\vec H}_0({\vec z})))_{{J}(p+1)}(  {\vec \theta}_{{J}(p+1)}  + ( \Pi (  {\vec \xi}))_{J(p)} {\vec e}_j \times {\vec \xi} )  \right ) \dif {\vec \xi}  ,\nonumber
\end{align}
and $| {\vec R}_{\mathrm{I}}({\vec x}) |  \le  C \nu \alpha^3 \alpha^{M}\| {\vec H}_0 \|_{W^{{M+1}, \infty} (B_\alpha)}$. Consequently, ${\vec R}_{\mathrm{I}}({\vec x})$ forms part of ${\vec R}({\vec x})$.
 \end{proof}
 
 \subsubsection{Approximation of $\mathrm{II}$}
In similar way to ~\cite{ammarivolkov2013}, we write $\mathrm{II}=\left ( 1- \frac{\mu_*}{\mu_0} \right )\left (\mathrm{II}_1+\mathrm{II}_2+\mathrm{II}_3+\mathrm{II}_4 \right )$ where
\begin{align*}
\mathrm{II}_1   = & - \int_{B_\alpha} {\vec D}_x^2 G({\vec x}, {\vec y}) \left ( {\vec H}_\alpha ({\vec y}) - \frac{\mu_0}{\mu_*} {\vec H}_0({\vec y}) - \frac{\mu_0}{\mu_*} {\vec H}_0^* \left ( \frac{ {\vec y}-{\vec z}}{\alpha} \right ) \right ) \dif {\vec y} ,\\
\mathrm{II}_2   = & -\frac{\mu_0}{\mu_*} \int_{B_\alpha}  \left ( {\vec D}_x^2 G({\vec x}, {\vec y}) -  
\sum_{m=0}^S  \frac{(-1)^m }{m!} ( {\vec D}_x^m ( {\vec D}_x^2 G ({\vec x},{\vec z})))_{K(m)}  ( \Pi ( {\vec y}-{\vec z} ) )_{K(m)}  \right )\nonumber \\
&\left (  {\vec H}_0({\vec y}) + {\vec H}_0^* \left ( \frac{ {\vec y}-{\vec z}}{\alpha} \right ) \right ) \dif {\vec y }  , \\
\mathrm{II}_3   =& -\frac{\mu_0}{\mu_*} \int_{B_\alpha} \sum_{m=0}^S  \frac{(-1)^m }{m!} ( {\vec D}_x^m ( {\vec D}_x^2 G ({\vec x},{\vec z})))_{K(m)} ( \Pi ( {\vec y}-{\vec z}) )_{K(m)} \nonumber\\
&\left  ( {\vec H}_0 ({\vec y} ) - \sum_{p=0}^T \frac{1 }{p!} ({\vec D}_z^p ({\vec H}_0({\vec z})))_{{J}(p+1)} ( \Pi ( {\vec y}-{\vec z})  )_{J(p)} {\vec e}_j \right ) \dif {\vec y}   , \\
\mathrm{II}_4   =& -\frac{\mu_0}{\mu_*} \int_{B_\alpha} \sum_{m=0}^S  \frac{(-1)^m }{m!} ( {\vec D}_x^m ( {\vec D}_x^2 G ({\vec x},{\vec z})))_{K(m)} ( \Pi ( {\vec y}-{\vec z} ) )_{K(m)} \nonumber \\
&\left   ( \sum_{p=0}^T \frac{1 }{p!} ({\vec D}_z^p ({\vec H}_0({\vec z})))_{{J}(p+1)} ( \Pi ( {\vec y}-{\vec z})  )_{J(p)} {\vec e}_j + {\vec H}_0^* \left ( \frac{ {\vec y}-{\vec z}}{\alpha} \right ) \right ) \dif {\vec y} ,
\end{align*}
and
\begin{equation}
{\vec H}_0^*( {\vec \xi}) = \frac{1}{\im \omega \mu_0} \nabla_\xi \times {\vec w}_0({\vec \xi}). \nonumber
\end{equation}

\begin{lemma} \label{lemma:boundII1II2II3}
We can bound $\mathrm{II}_1 $, $ \mathrm{II}_2 $ and $ \mathrm{II}_3 $ as
\begin{align}
|\mathrm{II}_1 |  \le & C \alpha^{4+P} \|  { \vec H}_0 \|_{W^{{P+1},\infty}(B_\alpha)} , \qquad 
|\mathrm{II}_2 \le  C  \alpha^{4+S}  \| {\vec H}_0 \|_ {W^{{P+1},\infty}(B_\alpha)} , \nonumber\\
|\mathrm{II}_3  | \le & C \alpha^{4+T} \| { \vec H}_0 \|_{W^{{T+1},\infty}(B_\alpha)}  \nonumber.
\end{align}
\end{lemma}
\begin{proof}
From Theorem~\ref{thm:revised31}, we have
\begin{align}
\left \|  {\vec H}_\alpha -\frac{\mu_0}{\mu_*} {\vec H}_0 -\frac{ \alpha}{\im \omega \mu_*} \nabla_x \times  {\vec w}_0 \left ( \frac{{\vec x} - {\vec z}}{\alpha} \right )  \right \|_{L^2(B_\alpha)} \le C &
 \left ( | 1 - \mu_r^{-1} | + \nu  \right ) \alpha^{\frac {5+2P}{2}} \|  { \vec H}_0 \|_{W^{{P+1},\infty}(B_\alpha)}, \nonumber
\end{align}
and so for $\mathrm{II}_1 $ we find that
\begin{align}
|\mathrm{II}_1 | \le & C \alpha^{\frac{3}{2}} \left ( | 1 - \mu_r^{-1} | + \nu  \right ) \alpha^{\frac {5+2P}{2}} \|  { \vec H}_0 \|_{W^{{P+1},\infty}(B_\alpha)} \nonumber \\
\le & C  \left ( | 1 - \mu_r^{-1} | + \nu  \right ) \alpha^{4+P} \|  { \vec H}_0 \|_{W^{{P+1},\infty}(B_\alpha)} \le  C   \alpha^{4+P} \|  { \vec H}_0 \|_{W^{{P+1},\infty}(B_\alpha)}. \nonumber
\end{align}
For $\mathrm{II}_2$, we have from (\ref{eqn:est4}) and~\cite{ammarivolkov2013} that
\begin{align}
|\mathrm{II}_2 | \le & C \alpha^{\frac{5+2S}{2}} \alpha^{ \frac{3}{2} } \| {\vec H}_0 \|_ {W^{{P+1},\infty}(B_\alpha)} \le  C  \alpha^{4+S}  \| {\vec H}_0 \|_ {W^{{P+1},\infty}(B_\alpha)} \nonumber .
\end{align}
Finally, for $|\mathrm{II}_3 |$, in a similar manner to (\ref{eqn:curlh0_curlfl2}) we have 
\begin{align}
|\mathrm{II}_3 | \le & C  \alpha^{\frac{5+2T}{2}} \alpha^{\frac{3}{2}} \| { \vec H}_0 \|_{W^{{T+1},\infty}(B_\alpha)} \le  C \alpha^{4+T} \| { \vec H}_0 \|_{W^{{T+1},\infty}(B_\alpha)} \nonumber .
\end{align}
\end{proof}

\begin{lemma} \label{lemma:boundII}
 The integral $\mathrm{II}$ can be expressed as 
\begin{align*}
\mathrm{II}= &
 \left ( 1 - \frac{\mu_0}{\mu_*} \right )  \sum_{m=0}^{M-1} \sum_{p=0}^{M-1-m}
\frac{ \alpha^{3+m+p} (-1)^m}{p! m!} ( {\vec D}_x^m ( {\vec D}_x^2 G ({\vec x},{\vec z})))_{[i,K(m+1)]} ({\vec e}_i \otimes {\vec e}_k)  \nonumber \\
& \int_B ( \Pi ( {\vec \xi}))_{K(m)}  ({\vec D}_z^p ({\vec H}_0({\vec z})))_{{J}(p+1)} \left ( \frac{1}{p+2} \nabla_\xi \times {\vec \theta}_{{J}(p+1) } +( \Pi ( {\vec \xi}))_{J(p)} {\vec e}_j  \right ) \dif {\vec \xi} 
 + {\vec R}({\vec x}),
 \end{align*}
 where $|{\vec R}({\vec x})| \le C\alpha^{3+M}   \| {\vec H}_0 \|_{W^{{M+1}, \infty} (B_\alpha)}  $  .
\end{lemma}
\begin{proof}
We recall that  $\mathrm{II}=\left ( 1- \frac{\mu_*}{\mu_0} \right )\left (\mathrm{II}_1+\mathrm{II}_2+\mathrm{II}_3+\mathrm{II}_4 \right )$ and, in light of Lemma~\ref{lemma:boundI}, we need to choose $P=M$. Making the choice of $S=M-1$ and $T=M$ we see from Lemma~\ref{lemma:boundII1II2II3} that the terms associated with $\mathrm{II}_1$, $\mathrm{II}_2$ and $\mathrm{II}_3$ all form part of ${\vec R}({\vec x})$ with conservative estimates in the power of $\alpha$ in $\mathrm{II}_1$ and $\mathrm{II}_3$, but all involving $\| {\vec H}_0 \|_{W^{{M+1}, \infty} (B_\alpha)}$.  Also
\begin{align*}
\mathrm{II}_4   =& -\frac{\mu_0}{\mu_*} \int_{B_\alpha} \sum_{m=0}^{M-1}  \frac{(-1)^m }{m!} ( {\vec D}_x^m ( {\vec D}_x^2 G ({\vec x},{\vec z})))_{K(m)} ( \Pi ( {\vec y}-{\vec z}))_{K(m)} \nonumber \\
&\left   ( \sum_{p=0}^{M} \frac{1 }{p!} ({\vec D}_z^p ({\vec H}_0({\vec z})))_{{J}(p+1)} ( \Pi ( {\vec y}-{\vec z} ) )_{J(p)} {\vec e}_j + {\vec H}_0^* \left ( \frac{ {\vec y}-{\vec z}}{\alpha} \right ) \right ) \dif {\vec y} \nonumber \\
= &  -\frac{\mu_0}{\mu_*} \alpha^3 \sum_{m=0}^{M-1} \sum_{p=0}^{M-1-m}  \frac{(-1)^m \alpha^{m+p} }{m! p! } ( {\vec D}_x^{2+m} G ({\vec x},{\vec z}) )_{[i,K(m+1)]}({\vec e}_i \otimes {\vec e}_k) \nonumber \\
& \int_B  (\Pi ( {\vec \xi}))_{K(m)} ({\vec D}_z^p ({\vec H}_0({\vec z})))_{{J}(p+1)}  \left ( 
\frac{1}{p+2} \nabla_\xi \times {\vec \theta}_{{J}(p+1)} +( \Pi  ( {\vec \xi})  )_{J(p)} {\vec e}_j   \right )  \dif {\vec \xi} +{\vec R}_{\mathrm{II}}( {\vec x} ),
\end{align*}
where
\begin{align*}
{\vec R}_{\mathrm{II}}( {\vec x} ) = & -\frac{\mu_0}{\mu_*} \alpha^3 \sum_{m=0}^{M-1} \sum_{p=M-m}^{M}  \frac{(-1)^m \alpha^{m+p} }{m! p! } ( {\vec D}_x^{2+m} G ({\vec x},{\vec z}))_{[i,K(m+1)]} ({\vec e}_i \otimes {\vec e}_k) \nonumber \\
& \int_B  (\Pi  ({\vec \xi}))_{K(m)} ({\vec D}_z^p ({\vec H}_0({\vec z})))_{{J}(p+1)}  \left ( 
\frac{1}{p+2} \nabla_\xi \times {\vec \theta}_{{J}(p+1)} + ( \Pi ( {\vec \xi})  )_{J(p)} {\vec e}_j   \right )  \dif {\vec \xi}\nonumber .
\end{align*}
It follows that $| {\vec R}_{\mathrm{II}}({\vec x}) |  \le  C  \alpha^3 \alpha^{M}\| {\vec H}_0 \|_{W^{{M+1}, \infty} (B_\alpha)}$ and so ${\vec R}_{\mathrm{II}}({\vec x})$ forms part of ${\vec R}({\vec x})$. Substitution of $\mathrm{II}_4$ in to the expression for $\mathrm{II}$ completes the proof.
\end{proof}

\section{Tensor representations} \label{sect:tensorrep}

\begin{lemma} \label{lemma:tensors}
The arrays of functions defined in (\ref{eqn:Arank4h}, \ref{eqn:Nrank2h}) are invariant under the orthogonal transformations
\begin{align}
  {\mathfrak A}_{[[i ,\ell, {K(m+1)],J(p+1)}]} [ {\mathcal J}(B) ] = &   {\mathcal J}_{i i'} {\mathcal J}_{\ell \ell '} {\mathcal J}_{k k'} {\mathcal J}_{j j'} {\mathfrak J}_{K(m) K'(m)} {\mathfrak J}_{J(p) J'(p)} {\mathfrak A}_{[[i',\ell ', {K'(m+1)],J'(p+1)}]}  [B]\nonumber , \\
  {\mathfrak N}_{{K(m+1)J(p+1) }} [{\mathcal J}(B)] = & {\mathcal J}_{k k'} {\mathcal J}_{j j'} {\mathfrak J}_{K(m) K'(m)} {\mathfrak J}_{J(p) J'(p)}   {\mathfrak N}_{{ K'(m+1)J'(p+1})}   [B] \nonumber ,
 \end{align}
 where the term inside the square parenthesis indicates the object for which the tensor is evaluated, ${\mathcal J}$ is an orthogonal transformation matrix and
 \begin{align*}
  {\mathfrak J}_{K(m) K'(m)}  :=\prod_{r=1}^m  {\mathcal J}_{k_r k_r'}, \qquad
  {\mathfrak J}_{J(p) J'(p)}  :=\prod_{r=1}^p  {\mathcal J}_{j_r j_r'}. \qquad
 \end{align*}
It follows that these arrays of functions are the coefficients of the  rank $ 4+ m+p$ and $2+m+p$ tensors,  ${\mathfrak A}$ and ${\mathfrak N}$, respectively.
\end{lemma}

\begin{proof}
Building on the previous results in Proposition 4.3 in~\cite{ammarivolkov2013b} and Theorem 3.1 in~\cite{ledgerlionheart2014}, we set ${\vec F}_{{\mathcal J}(B),{\vec e}_j,J(p)}$ to be the solution of
\begin{align*}
\nabla_\xi \times \mu_*^{-1} \nabla_\xi \times {\vec F}_{ {{\mathcal J}(B),{\vec e}_j,J(p)} } - \im \omega \sigma_* \alpha^2& {\vec F}_{{{\mathcal J}(B),{\vec e}_j,J(p)} }  =    \im \omega \sigma_* \alpha^{2} ( \Pi({\vec \xi}) )_{J(p)}  {\vec e}_j \times {\vec \xi}  && \text{in ${\mathcal J}(B)$ } , \\
\nabla_\xi \cdot {\vec F}_{ {{\mathcal J}(B),{\vec e}_j,J(p)} }  = & 0 && \text{in ${\mathbb R}^3 \setminus {\mathcal J}(B)$ } , \\
\nabla_\xi \times \mu_0^{-1} \nabla_\xi \times {\vec F}_{ {{\mathcal J}(B),{\vec e}_j,J(p)} }   = & {\vec 0} && \text{in ${\mathbb R}^3 \setminus {\mathcal J}(B)$ } , \\
[{\vec n} \times {\vec F}_{ {{\mathcal J}(B),{\vec e}_j,J(p)} } ]_{\partial {\mathcal J}(B)} =  & {\vec 0},  && \text{on $ \partial{\mathcal J}( B)$} , \\
 [{\vec n} \times \mu^{-1} \nabla_\xi  \times {\vec F}_{ {{\mathcal J}(B),{\vec e}_j,J(p)} } ]_{\partial {\mathcal J}(B)}  = &  - (p+2)[ \mu^{-1} ]_\Gamma  {\vec n} \times {\vec e}_j  (\Pi({\vec \xi})  )_{J(p)} 
 && \text{on $\partial {\mathcal J}(B)$}, \\
 \int_{\partial {\mathcal J}(B)} {\vec n} \cdot {\vec F}_{ {{\mathcal J}(B),{\vec e}_j,J(p)} } |_+  \dif {\vec \xi}   =& 0 , \\
{\vec F}_{ {{\mathcal J}(B),{\vec e}_j,J(p)} }  = & O( | {\vec \xi} |^{-1}) && \text{as $|{\vec \xi} | \to \infty$ },
\end{align*} 
and, by following similar arguments to~\cite{ammarivolkov2013b}, we find that
\begin{equation*}
{\vec F}_{ {{\mathcal J}(B),{\vec e}_j,J(p)} } = |{\mathcal J} |  {\mathfrak J}_{J(p) J'(p)} {\mathcal J} {\vec F}_{ {  B,{\mathcal J}^T {\vec e}_j,J'(p)} } .
\end{equation*}
Then, by writing $A(p,m)= - \im \nu \frac{(-1)^m \alpha^{3+p+m}}{p!(m+1)! (p+2)} $, we have 
\begin{align}
{} &  {{\mathfrak A}_{[[i,\ell, {K}(m+1)],{J}(p+1)]}  [ {\mathcal J} ( B ) ] } =  A(p,m) {\vec e}_i \cdot \nonumber\\
&\int_{ {\mathcal J} (B)}   {\vec e}_k \times 
( {\xi}_{\ell }  (\Pi({\vec \xi}))_{K(m)}   (  {\vec F}_{ {\mathcal J}(B),{\vec e}_j,J(p) }   + (\Pi({\vec \xi}))_{J(p)} {\vec e}_j \times {\vec \xi} ))   \dif {\vec \xi}  \nonumber \\
 = &  A(p,m) {\vec e}_i \cdot \int_B {\vec e}_k \times (  {\mathcal J}_{\ell \ell '} {\xi}_{\ell ' }   {\mathfrak J}_{ K '(m) K(m)} (\Pi({\vec \xi}))_{K'(m)}  ( |{\mathcal J} |  {\mathfrak J}_{J(p) J'(p)} {\mathcal J} {\vec F}_{ {  B,{\mathcal J}^T {\vec e}_j,J'(p)} } + \nonumber \\
 &\qquad \qquad \qquad {\mathfrak J}_{J(p) J'(p)} (\Pi({\vec \xi}))_{J'(p)} {\vec e}_j \times ({\mathcal J} {\vec \xi}) )) \dif {\vec \xi}\nonumber \\
 = & |J| {\mathcal J}_{\ell \ell '}  {\mathfrak J}_{J(p) J'(p)}  {\mathfrak J}_{ K '(m) K(m)} A(p,m)   {\vec e}_i \cdot \int_B {\vec e}_k \times (  {\mathcal J}( {\xi}_{\ell ' }(\Pi({\vec \xi}))_{K'(m)} ({\vec F}_{ {  B,{\mathcal J}^T {\vec e}_j,J'(p)} } + \nonumber \\
 &\qquad \qquad \qquad  (\Pi({\vec \xi}))_{J'(p)} ({\mathcal J}^T {\vec e}_j) \times {\vec \xi}) ) )\dif {\vec \xi}\nonumber \\
 = & |J|^2 {\mathcal J}_{\ell \ell '}  {\mathfrak J}_{J(p) J'(p)}  {\mathfrak J}_{ K '(m) K(m)} A(p,m)   {\vec e}_i \cdot \int_B {\mathcal J}(   ({\mathcal J}^T {\vec e}_k) \times (  {\xi}_{\ell ' }(\Pi({\vec \xi}))_{K'(m)} ( {\vec F}_{ {  B,{\mathcal J}^T {\vec e}_j,J'(p)} } + \nonumber \\
 &\qquad \qquad \qquad  (\Pi({\vec \xi}))_{J'(p)} ({\mathcal J}^T {\vec e}_j) \times {\vec \xi}) ) )\dif {\vec \xi}\nonumber \\
= &  {\mathcal J}_{\ell \ell '} {\mathcal J}_{k k '} {\mathcal J}_{i i '}  {\mathfrak J}_{J(p) J'(p)}  {\mathfrak J}_{ K '(m) K(m)} A(p,m)   {\vec e}_{i'} \cdot \int_B    {\vec e}_{k'} \times (  {\xi}_{\ell ' }(\Pi({\vec \xi}))_{K'(m)}  ({\vec F}_{ {  B,{\mathcal J}^T {\vec e}_j,J'(p)} } + \nonumber \\
 &\qquad \qquad \qquad  (\Pi({\vec \xi}))_{J'(p)} ({\mathcal J}^T {\vec e}_j) \times {\vec \xi}))  \dif {\vec \xi}\nonumber \\
= &  {\mathcal J}_{\ell \ell '} {\mathcal J}_{k k '} {\mathcal J}_{i i '} {\mathcal J}_{j j '}  {\mathfrak J}_{J(p) J'(p)}  {\mathfrak J}_{ K '(m) K(m)} A(p,m)   {\vec e}_{i'} \cdot \int_B    {\vec e}_{k'} \times (  {\xi}_{\ell ' }(\Pi({\vec \xi}))_{K'(m)} ( {\vec F}_{ {  B, {\vec e}_{j'},J'(p)} } + \nonumber \\
 &\qquad \qquad \qquad  (\Pi({\vec \xi}))_{J'(p)} {\vec e}_{j'} \times {\vec \xi}))  \dif {\vec \xi}\nonumber \\
  = &  {\mathcal J}_{\ell \ell '} {\mathcal J}_{k k '} {\mathcal J}_{i i '} {\mathcal J}_{j j '}  {\mathfrak J}_{J(p) J'(p)}  {\mathfrak J}_{ K '(m) K(m)} { {\mathfrak A}_{[[i',\ell ', {K} '(m+1)],{J}'(p+1)]} [ B  ]}\nonumber ,
\end{align}
as desired. Similarly, by using
\begin{equation*}
\nabla_\xi \times  {\vec F}_{ {{\mathcal J}(B),{\vec e}_j,J(p)} }={\mathfrak J}_{J(p)J'(p)} {\mathcal J} \nabla_\xi \times ( {\vec F}_{  (B), ( {\mathcal J}^T{\vec e}_j) ,J'(p)} ) ,
 \end{equation*}
 we find that
 \begin{align}
{}&  {\mathfrak N}_{{K}(m+1) {J}(p+1) } [{\mathcal J}(B)] =   N(p,m) {\vec e}_k \cdot \nonumber \\
 &\int_{{\mathcal J}(B)}  ( \Pi({\vec \xi}))_{K(m)}   \left ( \frac{1}{p+2}  \nabla_\xi \times  {\vec F}_{ {{\mathcal J}(B),{\vec e}_j,J(p)} } +( \Pi({\vec \xi}))_{J(p)} {\vec e}_j  \right ) \dif {\vec \xi} \nonumber \\
 = &  N(p,m) {\vec e}_k \cdot
 \int_{B}  {\mathfrak J}_{ K(m) K'(m)} ( \Pi({\vec \xi}))_{K'(m)}   \left ( \frac{1}{p+2} {\mathfrak J}_{J(p)J'(p)} {\mathcal J} \nabla_\xi \times  {\vec F}_{  B,  {\mathcal J}^T{\vec e}_j ,J'(p)} 
  +\right . \nonumber \\
  &\qquad \qquad \qquad \left .  {\mathfrak J}_{J(p) J'(p)} ( \Pi({\vec \xi}))_{J'(p)} {\vec e}_j  \right ) \dif {\vec \xi}\nonumber \\
 = & {\mathcal J}_{k k'}  {\mathfrak J}_{ K(m) K'(m)}  {\mathfrak J}_{J(p) J'(p)} N(p,m) {\vec e}_{k'} \cdot \int_{B} ( \Pi({\vec \xi}))_{K'(m)}   \left ( \frac{1}{p+2} \nabla_\xi \times  {\vec F}_{  B,  {\mathcal J}^T{\vec e}_j ,J'(p)}    +\right . \nonumber \\
  &\qquad \qquad \qquad \left .  ( \Pi({\vec \xi}))_{J'(p)} {\mathcal J}_{j j'} {\vec e}_{j'}  \right ) \dif {\vec \xi}\nonumber \\
  = & {\mathcal J}_{k k'}  {\mathcal J}_{j j'}  {\mathfrak J}_{ K(m) K'(m)}  {\mathfrak J}_{J(p) J'(p)}   N(p,m) {\vec e}_{k'} \cdot \int_{B} ( \Pi({\vec \xi}))_{K'(m)}   \left ( \frac{1}{p+2} \nabla_\xi \times  {\vec F}_{  B, {\vec e}_{j'} ,J'(p)}   +\right . \nonumber \\
  &\qquad \qquad \qquad \left .  ( \Pi({\vec \xi}))_{J'(p)} {\vec e}_{j'}  \right ) \dif {\vec \xi}\nonumber \\
  = & {\mathcal J}_{k k'}  {\mathcal J}_{j j'}  {\mathfrak J}_{ K(m) K'(m)}  {\mathfrak J}_{J(p) J'(p)}  {\mathfrak N}_{{K}'(m+1) {J}'(p+1)} [ B] \nonumber ,
\end{align}
 where $N(p,m):=\frac{ (-1)^m\alpha^{3+m+p} }{p! m!} \left ( 1 - \frac{\mu_0}{\mu_*} \right )$.
\end{proof}

\begin{corollary}
Note that an alternative transmission problem
\begin{align*}
\nabla_\xi \times \mu_*^{-1} \nabla_\xi \times \tilde{\vec F}_{ {{\mathcal J}(B),J(p)} } - \im \omega \sigma_* \alpha^2 \tilde{\vec F}_{{{\mathcal J}(B),J(p)} }  = &   \im \omega \sigma_* \alpha^{2} ( \Pi({\vec \xi}) )_{J(p)}  {\vec \xi}  && \text{in ${\mathcal J}(B)$ } , \\
\nabla_\xi \cdot \tilde{\vec F}_{ {{\mathcal J}(B), J(p)} }  = & 0 && \text{in ${\mathbb R}^3 \setminus {\mathcal J}(B)$ } , \\
\nabla_\xi \times \mu_0^{-1} \nabla_\xi \times \tilde{\vec F}_{ {{\mathcal J}(B), J(p)} }   = & {\vec 0} && \text{in ${\mathbb R}^3 \setminus {\mathcal J}(B)$ } , \\
[{\vec n} \times \tilde{\vec F}_{ {{\mathcal J}(B),J(p)} } ]_{\partial {\mathcal J}(B)} =  & {\vec 0},  && \text{on $ \partial{\mathcal J}( B)$} , \\
 [{\vec n} \times \mu^{-1} \nabla_\xi  \times \tilde{\vec F}_{ {{\mathcal J}(B),J(p)} } ]_{\partial {\mathcal J}(B)}  = &  (p+2)[ \mu^{-1} ]_\Gamma  {\vec n}  (\Pi({\vec \xi})  )_{J(p)} 
 && \text{on $\partial {\mathcal J}(B)$}, \\
 \int_{\partial {\mathcal J}(B)} {\vec n} \cdot \tilde{\vec F}_{ {{\mathcal J}(B),J(p)} } |_+  \dif {\vec \xi}   =& 0 , \\
\tilde{\vec F}_{ {{\mathcal J}(B),J(p)} }  = & O( | {\vec \xi} |^{-1}) && \text{as $|{\vec \xi} | \to \infty$ },
\end{align*} 
satisfying $ {\vec F}_{ {{\mathcal J}(B),{\vec e}_j , J(p)} } =  {\vec e}_j \times \tilde{\vec F}_{ {{\mathcal J}(B),J(p)} }$ can be introduced. The advantage of the formulation for  $\tilde{\vec F}_{ {{\mathcal J}(B),J(p)} }$ is that it and obeys the simpler transformation
\begin{equation*}
\tilde{\vec F}_{ {{\mathcal J}(B),J(p)} } =   {\mathfrak J}_{J(p) J'(p)} {\mathcal J} \tilde{\vec F}_{ {  B,  J'(p)} } ,
\end{equation*}
which is consistent with a rank $2+p$ tensor. Nonetheless, the transformations of the components of ${\mathfrak A}$ and ${\mathfrak N}$, if written in terms of $\tilde{\vec F}_{ {{\mathcal J}(B),J(p)} }$, remain unchanged. However, we prefer to continue use  $ {\vec F}_{ {{\mathcal J}(B),{\vec e}_j , J(p)} }$, and hence  ${\vec \theta}_{{J}(p+1)}$, since it results in a simpler form of ${\mathfrak N}_{{K}(m+1) {J}(p+1) }$.
\end{corollary}

Having verified that  ${\mathfrak A}$ and ${\mathfrak N}$ are tensors we now investigate whether the former rank  $4+m+p$  tensor can instead be represented by a  rank $2+m+p$  tensor.

\subsection{Reduction of ${\mathfrak A}$}
Without loss of generality, we assume a positively orientated orthogonal frame. Any change in sign that follows from the reduction in rank for a different frame will cancel as we shall reduce the rank of ${\mathfrak A}$ by 2.

\begin{lemma} \label{lemma:reduction1}
The coefficients of  ${\mathfrak A}$ satisfy {$   {\mathfrak A}_{[[i,\ell, k,K(m)],J(p+1)]}= -  {\mathfrak A}_{[[k, \ell ,i ,K(m)],J(p+1)]}$}  and so it is possible to reduce the rank of   ${\mathfrak A}$ by one and to represent it by the rank $3 +p+m$  tensor density with coefficients 
\begin{align}
{{\mathfrak C}_{[[r,\ell,K(m)] ,J(p+1)]} }: = & { \widecheck{\mathfrak A}_{[[r,\ell,K(m)] ,J(p+1)]} =\dfrac{1}{2} \varepsilon_{ r i k}  {\mathfrak A}_{[[i,\ell, k,K(m)],J(p+1)]}   }\nonumber \\
= & 
- \im \nu \frac{(-1)^m \alpha^{3+p+m}}{p!(m+1)! (p+2)}  {\vec e}_r \cdot \int_B    
 {\xi}_{\ell }  (\Pi({\vec \xi}))_{K(m)}   (   {\vec \theta}_{{J}(p+1)}  + (\Pi({\vec \xi}))_{J(p)} {\vec e}_j \times {\vec \xi} )   \dif {\vec \xi} , \nonumber
\end{align}
where we note that $ {{\mathfrak A}_{[[i,\ell, k,K(m)],J(p+1)]]} = \varepsilon_{i k r} \widecheck{\mathfrak A}_{[[r, \ell,K(m)],J(p+1)]}}$.
\end{lemma}

\begin{proof}
We first write $A(p,m)= - \im \nu \frac{(-1)^m \alpha^{3+p+m}}{p!(m+1)! (p+2)} $ so that
\begin{align}
 { {\mathfrak A}_{[[i,\ell, k,{K}(m)],{J}(p+1)]} }= & A(p,m) {\vec e}_i \cdot \int_B   {\vec e}_k \times 
( {\xi}_{\ell }  (\Pi({\vec \xi}))_{K(m)}   (   {\vec \theta}_{{J}(p+1)}  + (\Pi({\vec \xi}))_{J(p)} {\vec e}_j \times {\vec \xi} ))   \dif {\vec \xi}  \nonumber \\
 = & A(p,m) {\vec e}_i \cdot {\vec e}_k \times \left (  \int_B   
 {\xi}_{\ell }  (\Pi({\vec \xi}))_{K(m)}   (   {\vec \theta}_{{J}(p+1)}  + (\Pi({\vec \xi}))_{J(p)} {\vec e}_j \times {\vec \xi} )   \dif {\vec \xi} \right ) \nonumber \\
= & -A(p,m) {\vec e}_k \cdot {\vec e}_i \times \left (  \int_B   
  {\xi}_{\ell }  (\Pi({\vec \xi}))_{K(m)}   (   {\vec \theta}_{{J}(p+1)}  +( \Pi({\vec \xi}))_{J(p)} {\vec e}_j \times {\vec \xi} )   \dif {\vec \xi} \right ) \nonumber \\
 = &{-{\mathfrak A}_{[[k, \ell ,i , K(m)],J(p+1)]}} \nonumber ,
 \end{align}
and the result then immediately follows by similar operations to  Lemma 4.1 in~\cite{ledgerlionheart2014}. 
\end{proof}

\begin{lemma} \label{lemma:reduction2}
The coefficients of the tensor density ${\mathfrak C}$ satisfy {${\mathfrak C}_{[[r,\ell,K(m)] ,J(p+1)]} = -{\mathfrak C}_{[[\ell,r,K(m)] ,J(p+1)]}$}, under summation with $( {\vec D}^{2+m} G( {\vec x},{\vec z} )  )_{[\ell ,K(m+1)]} $ and   $ ({\vec D}^p({\vec H}({\vec z})))_{ {J}(p+1)}$,  and so we can reduce the rank of  ${\mathfrak C}$ by one and represent it by the rank $2+m+p$ tensor with coefficients
\begin{align}
  \widecheck{\mathfrak C}_{ {K}(m+1) {J}(p+1)}= &{  \widecheck{\mathfrak C}_{[[k,K(m)] ,J(p+1)]} = \dfrac{1}{2} \varepsilon_{ k    \ell r }  {\mathfrak C}_{[ [r,\ell ,K(m)],J(p+1)]} }\nonumber\\
= & - \im \nu \frac{(-1)^m \alpha^{3+p+m}}{2p!(m+1)! (p+2)}  {\vec e}_k \cdot \int_B    
  {\vec \xi} \times (  (\Pi({\vec \xi}))_{K(m)}   (   {\vec \theta}_{{J}(p+1)}  + (\Pi({\vec \xi}))_{J(p)} {\vec e}_j \times {\vec \xi} )   ) \dif {\vec \xi} , \nonumber
\end{align}
where we note that $ {{\mathfrak C}_{[[r,\ell, K(m)],J(p+1)]} = \varepsilon_{\ell r k} \widecheck{\mathfrak C}_{[[k,K(m) ],J(p+1)]}}= \varepsilon_{\ell r k} \widecheck{\mathfrak C}_{K(m+1) J(p+1)} $.

\end{lemma}

\begin{proof}
We can use the transmission problem (\ref{eqn:transproblemtheta}) to  write 
\begin{align}
{{\mathfrak C}_{[[r,\ell,K(m)] ,J(p+1)]}} = & A(m,p) {\vec e}_r \cdot \int_B    
  {\xi}_{\ell }  (\Pi({\vec \xi}))_{K(m)}   (   {\vec \theta}_{ {J}(p+1)}  + (\Pi({\vec \xi}))_{J(p)} {\vec e}_j \times {\vec \xi} )   \dif {\vec \xi}  \nonumber \\
 = & \frac{A(m,p)}{\im \nu} \int_B ( {{\vec e}_r { \xi}_{\ell }  (\Pi({\vec \xi}))_{K(m)} )  \cdot \nabla_\xi \times \mu_*^{-1} \nabla_\xi \times {\vec \theta}_{{J}(p+1)}} \dif {\vec {\xi}}
 =  \frac{A(m,p)}{\im \nu}  \mathrm{T} , \nonumber
 \end{align}
where, by application of integration by parts, we have $ \mathrm{T}= \mathrm{T}_1 + \mathrm{T}_2 + \mathrm{T}_3$ and
\begin{align}
\mathrm{T}_1 =&  -\int_B ( \Pi({\vec \xi}) )_{K(m)}  ( {\vec e}_\ell  \times {\vec e}_r ) \cdot  \mu_*^{-1} \nabla_\xi \times {\vec \theta}_{{J}(p+1)} \dif {\vec {\xi}} \nonumber , \\
\mathrm{T}_2 =&   - \int_B   { \xi}_{\ell } ( \nabla _\xi(  (\Pi({\vec \xi}))_{K(m)}  ) \times {\vec e}_r ) \cdot  \mu_*^{-1} \nabla_\xi \times {\vec \theta}_{{J}(p+1)} \dif {\vec {\xi}} \nonumber , \\
 \mathrm{T}_3 =& \int_\Gamma \left .  {\vec n}^- \cdot  \mu_* \nabla_\xi \times {\vec \theta}_{{J}( p+1)}  \times ( {\vec e}_r { \xi}_{\ell }  ( \Pi({\vec \xi}) )_{K(m)} ) \right |_- \dif {\vec \xi} \nonumber .
\end{align}
We see immediately that
\begin{equation}
\mathrm{T}_1 =  - \varepsilon_{s r \ell} {\vec e}_s  \cdot \int_B ( \Pi({\vec \xi}) )_{K(m)}    \mu_*^{-1} \nabla_\xi \times {\vec \theta}_{{J}(p+1)} \dif {\vec {\xi}} , \nonumber
\end{equation}
and so is skew symmetric with respect to indices $r$ and $\ell$. 
In light of (\ref{eqn:fulltensorexp}), we see $\mathrm{T}_1$ is summed with $( {\vec D}_x^{2+m} G( {\vec x},{\vec z} )  )_{[\ell ,K(m+1)]} $ and  $ ( {\vec D}_z^p({\vec H}({\vec z})))_{ {J}(p+1)}$ and so we have
\begin{align}
&( {\vec D}_x^{2+m} G( {\vec x},{\vec z} )  )_{[\ell ,K(m+1)]} \mathrm{T}_1  (  {\vec D}_z^p({\vec H}({\vec z})))_{ {J}(p+1)} =\nonumber \\
  &-( {\vec D}_x^{2+m} G( {\vec x},{\vec z} )  )_{[\ell, K(m+1)]}  \varepsilon_{s r \ell} {\vec e}_s  \cdot \int_B ( \Pi({\vec \xi}) )_{K(m)}    \mu_*^{-1} \nabla_\xi \times {\vec \theta}_{{J}(p+1)} \dif {\vec {\xi}} (  {\vec D}_z^p({\vec H}({\vec z})))_{ {J}(p+1)} . \nonumber
\end{align}

In light of (\ref{eqn:fulltensorexp}) and the form of $\mathrm{T}_2$ we see that the term ${\xi}_\ell (\nabla_\xi (  (\Pi({\vec \xi}) )_{K(m)}  ) \times {\vec e}_r )$ will be summed with  $( {\vec D}_x^{2+m} G( {\vec x},{\vec z} )  )_{[\ell ,K(m+1)]} $ and so
\begin{align}
( {\vec D}_x^{2+m} & G( {\vec x},{\vec z} )  )_{[\ell ,K(m+1)]} { \xi}_\ell (\nabla_\xi (  (\Pi({\vec \xi}) )_{K(m)}  ) \times {\vec e}_r )  = \varepsilon_{str}
(  {\vec D}_x^{2+m} G( {\vec x},{\vec z} )  )_{[\ell ,K(m+1)]}  {\xi}_\ell  \frac{\partial}{{\partial_{ \xi_t}}}(  (\Pi({\vec \xi}) )_{K(m)}  ) {\vec e}_s \nonumber \\
  =&  \varepsilon_{str}
 ( {\vec D}_x^{2+m} G( {\vec x},{\vec z} )  )_{[\ell ,K(m+1)]}  {\xi}_\ell  \left ( 
  \delta_{t k_1} \xi_{k_2} \cdots \xi_{k_m}  + \cdots + \xi_{k_1} \cdots \xi_{k_{m-1}} \delta_{t k_m} \right ){\vec e}_s \nonumber \\
  =&  \varepsilon_{str} \left (
 ( {\vec D}_x^{2+m} G( {\vec x},{\vec z} )  )_{[\ell ,k,t, k_2,\cdots, k_m]}  {\xi}_\ell \xi_{k_2} \cdots \xi_{k_m}  + \right .\cdots \nonumber \\
 &+ \left .  (   {\vec D}_x^{2+m} G( {\vec x},{\vec z} )  )_{[\ell ,k,k_1, \cdots, k_{m-1},t]}{ \xi}_{k_1} \xi_{k_2} \cdots \xi_{k_m-1} \xi_\ell \right) {\vec e}_s
 \nonumber \\
 = & m \varepsilon_{str}  ({\vec D}_x^{2+m} G( {\vec x},{\vec z} )  )_{[k,t, K(m) ]}  (\Pi({\vec \xi}))_{K(m)}{\vec e}_s =m \varepsilon_{s\ell r}  ({\vec D}_x^{2+m} G( {\vec x},{\vec z} )  )_{[\ell , K(m+1) ]}  (\Pi({\vec \xi}))_{K(m)}{\vec e}_s \nonumber,
\end{align}
 by interchanging the order of differentiation in $ {\vec D}_x^{2+m} G( {\vec x},{\vec z} )  $. Thus,  this term is skew symmetric with respect to indices $r$ and $\ell$,
 and, in a similar manner to $\mathrm{T}_1$
\begin{align}
&( {\vec D}_x^{2+m} G( {\vec x},{\vec z} )  )_{[\ell ,K(m+1)]} \mathrm{T}_2    {\vec D}_z^p({\vec H}({\vec z}))_{ {J}(p+1)} \nonumber \\
&=  -m( {\vec D}_x^{2+m} G( {\vec x},{\vec z} )  )_{[\ell ,K(m+1)]}  \varepsilon_{s r \ell} {\vec e}_s  \cdot \int_B ( \Pi({\vec \xi}) )_{K(m)}    \mu_*^{-1} \nabla_\xi \times {\vec \theta}_{{J}(p+1)} \dif {\vec {\xi}}   ({\vec D}_z^p({\vec H}({\vec z}) )) _{ {J}(p+1)} . \nonumber
\end{align}
 For $ \mathrm{T}_3 $ we use the alternative form
 \begin{align}
 \mathrm{T}_3 =  & \int_\Gamma \left ( {\vec e}_r \xi_\ell ( \Pi({\vec \xi}))_{K(m)} \right ) \cdot \left ( \left . {\vec n}^- \times \mu_0^{-1} \nabla_\xi \times {\vec \theta}_{{J}( p+1)}  \right |_+    + (p+2) [ \mu^{-1}]_\Gamma {\vec n}^- \times ((\Pi({\vec \xi}))_{J(p)} {\vec e}_j  )  \right )\dif {\vec \xi} \nonumber \\
  = & - \int_{B^c}  \nabla_\xi \cdot ( \mu_0^{-1} \nabla_\xi \times {\vec \theta}_{{J}(p+1)}  \times ( {\vec e}_r \xi_\ell ( \Pi({\vec \xi}))_{K(m)} ) )
\dif {\vec \xi} \nonumber \\
& +  (p+2) [ \mu^{-1}]_\Gamma \int_B \nabla_\xi \cdot   (   ( (\Pi({\vec \xi}))_{J(p)} {\vec e}_j ) \times   \left ( {\vec e}_r \xi_\ell ( \Pi({\vec \xi}))_{K(m)} \right )  ) \dif {\vec \xi} \nonumber \\
 = &- \int_{B^c}   \mu_0^{-1} \nabla_\xi \times {\vec \theta}_{{J}(p+1)}  \cdot  {\vec e}_\ell \times {\vec e}_r  ( \Pi({\vec \xi}))_{K(m)} \dif {\vec \xi}   
 - \int_{B^c}   \mu_0^{-1} \nabla_\xi \times {\vec \theta}_{{J}(p+1)}  \cdot ( \nabla_\xi ( \Pi({\vec \xi}))_{K(m)}   \times {\vec e}_r \xi_\ell )  \dif {\vec \xi}   \nonumber \\
 &+(p+2) [ \mu^{-1}]_\Gamma \int_B ( \nabla_\xi    ( \Pi({\vec \xi}))_{J(p)} \times  {\vec e}_j ) \cdot   \left ( {\vec e}_r \xi_\ell ( \Pi({\vec \xi}))_{K(m)} \right )   \dif {\vec \xi} \nonumber \\
 &+(p+2) [ \mu^{-1}]_\Gamma \int_B ( (  \Pi({\vec \xi}))_{J(p)}  {\vec e}_j ) \cdot  ( {\vec e}_\ell \times  {\vec e}_r  ( \Pi({\vec \xi}))_{K(m)} )   \dif {\vec \xi} \nonumber \\
 &+(p+2) [ \mu^{-1}]_\Gamma \int_B ( (  \Pi({\vec \xi}))_{J(p)}  {\vec e}_j ) \cdot  (\nabla_\xi ( \Pi({\vec \xi}))_{K(m)}  \times  {\vec e}_r) \xi_\ell    \dif {\vec \xi} \nonumber \\
 = & \mathrm{T}_3^A +\mathrm{T}_3^B + \mathrm{T}_3^C +   \mathrm{T}_3^D+\mathrm{T}_3^E , \nonumber
   \end{align}
 by using integration by parts and the transmission problem (\ref{eqn:transproblemtheta}). We see immediately that $\mathrm{T}_3^A$ and $ \mathrm{T}_3^D$ are skew symmetric with respect to $r$ and $\ell$. Terms $\mathrm{T}_3^B$ and $\mathrm{T}_3^E$ are similar to $\mathrm{T}_2$ and can be treated analogously. In term $\mathrm{T}_3^C$ note this is summed with $ ({\vec D}_z^p({\vec H}_0({\vec z})))_{ {J}(p+1)}$ and so recalling (\ref{eqn:curlXizero}) 
 \begin{align} 
(  \nabla_\xi ( \Pi({\vec \xi}))_{J(p)} \times  {\vec e}_j ) ({\vec D}_z^p({\vec H}_0({\vec z})))_{ {J}(p+1)}  = \nabla_\xi \times( (\Pi({\vec \xi}))_{J(p)}   ({\vec D}_z^p({\vec H}_0({\vec z})))_{ {J}(p+1)} {\vec e}_j ) = {\vec 0} \nonumber ,
  \end{align}
in $B$ and consequently  $\mathrm{T}_3^C   ( {\vec D}_z^p({\vec H}_0({\vec z})))_{ {J}(p+1)}=0$. 
\end{proof}

\section*{Acknowledgements}
The first author would like to acknowledge the support received from the EPSRC grant EP/R002134/1 and the second author would like to acknowledge the support received from the EPSRC grant  EP/R002177/1 and  a Royal Society Wolfson Research Merit Award and a Royal Society Challenge Grant. {We would also like to thank the anonymous reviewer for making a number of useful suggestions, which have improved the article.}

\bibliographystyle{plain}
\bibliography{ledgerlionheart}

\begin{thebibliography}{10}

\bibitem{ammaribuffa2000}
H.~Ammari, A.~Buffa, and J.~C. {N\'ed\'elec}.
\newblock A justification of eddy currents model for the {M}axwell equations.
\newblock {\em SIAM J. Appl. Math.}, 60:1805--1823, 2000.

\bibitem{ammarivolkov2013}
H.~Ammari, J.~Chen, Z.~Chen, J~Garnier, and D.~Volkov.
\newblock Target detection and characterization from electromagnetic induction
  data.
\newblock {\em J. {Math.} Pure. {Appl.}}, 101:54--75, 2014.

\bibitem{ammarivolkov2013b}
H.~Ammari, J.~Chen, Z.~Chen, D.~Vollkov, and H.~Wang.
\newblock Detection and classification from electromagnetic induction data.
\newblock {\em J. Comput. Phys.}, 301:201--217, 2015.

\bibitem{ammarikang2003}
H.~Ammari and H.~Kang.
\newblock High order terms in the asymptotic expansions of the steady state
  voltage potentials in the presence of conductivity inhomogeneities of small
  diameter.
\newblock {\em SIAM J. Math. Anal.}, 34:1152--1166, 2003.

\bibitem{ammarikanghelm}
H.~Ammari and H.~Kang.
\newblock Boundary layer techniques for solving the helmholtz equation in the
  presence of small inhomogeneities.
\newblock {\em Journal of Mathematical Analysis and Applications},
  296:190--208, 2004.

\bibitem{ammarikanglecturenotes}
H.~Ammari and H.~Kang.
\newblock {\em Reconstruction of Small Inhomogeneities from Boundary
  Measurements, Lecture Notes in Mathematics}, volume 1846.
\newblock Springer-Verlag, Berlin, 2004.

\bibitem{ammarikangbook}
H.~Ammari and H.~Kang.
\newblock {\em Polarization and Moment Tensors: With Applications to Inverse
  Problems}.
\newblock Springer, 2007.

\bibitem{ammarivogeluisvolkov}
H.~Ammari, M.~S. Vogelius, and D.~Volkov.
\newblock Asymptotic formulas for the perturbations in the electromagnetic
  fields due to the presence of inhomogeneities of small diamater ii. the full
  {M}axwell equations.
\newblock {\em J. Math. Pure. Appl.}, 80:769--814, 2001.

\bibitem{ammarivolkov2005}
H.~Ammari and D.~Volkov.
\newblock The leading-order term in the asymptotic expansion of the scattering
  amplitude of a collection of finite number dielectric inhomogeneities of
  small diameter.
\newblock {\em International Journal for Multiscale Computational Engineering},
  3:149--160, 2005.

\bibitem{barrowes2004}
B.~E. Barrowes, K.~O'Neill, T.M. Gregorcyzk, and J.~A. Kong.
\newblock Broadband analytical magnetoquasistatic electromagnetic induction
  solution for a conducting and permeable spheroid.
\newblock {\em IEEE Trans. Geosci. Remote Sens.}, 42:2479--2489, 2004.

\bibitem{bossavit1998}
A.~Bossavit.
\newblock Magnetostatic problems in multiply connected regions: some properties
  of the curl operator.
\newblock {\em IEE Proceedings}, 135:179--187, 1988.

\bibitem{ledgerlionheart2014}
P.~D. Ledger and W.~R.~B. Lionheart.
\newblock Characterising the shape and material properties of hidden targets
  from magnetic induction data.
\newblock {\em IMA J. Appl. Math.}, 80:1776--1798, 2015.

\bibitem{ledgerlionheart2012}
P.~D. Ledger and W.~R.~B. Lionheart.
\newblock The perturbation of electromagnetics fields at distances that are
  large compared with the object's size.
\newblock {\em IMA J. Appl. Math.}, 80:865--892, 2015.

\bibitem{ledgerlionheart2016}
P.~D. Ledger and W.~R.~B. Lionheart.
\newblock Understanding the magnetic polarizability tensor.
\newblock {\em IEEE Transactions on Magnetics}, 52:6201216, 2016.

\bibitem{ledgerlionheart2018}
P.~D. Ledger and W.~R.~B. Lionheart.
\newblock An explicit formula for the magnetic polarizability tensor for object
  characterisation.
\newblock {\em Submitted}, 2018.

\bibitem{peyton2013}
L.~A. Marsh, C.~Ktisis, A.~J{\"a}rvi, D.~Armitage, and A.~J. Peyton.
\newblock Three-dimensional object location and inversion of the magnetic
  polarisability tensor at a single frequency using a walk-through metal
  detector.
\newblock {\em Meas. Sci. Technol.}, 24:045102, 2013.

\bibitem{norton2001}
S.~J. Norton and I.~J. Won.
\newblock Identification of buried unexploded ordance from broadband induction
  data.
\newblock {\em IEEE Trans. Geosci. Remote Sens.}, 39:2253--2261, 2001.

\end{thebibliography}

\end{document}